\documentclass[12pt,english,a4paper]{article}
\usepackage{amssymb}
\usepackage{amsmath}
\usepackage{epsfig}
\usepackage{graphicx}

\usepackage{amsthm}
\usepackage{amsfonts}
\usepackage{latexsym}

\usepackage[T1]{fontenc}
\usepackage[utf8]{inputenc}
\usepackage{babel}
\usepackage{xcolor}

\usepackage{enumerate}

\newcommand{\Rd}{{\mathbb R}^d}

\newcommand{\R}{\mathbb{R}}
\newcommand{\C}{\mathbb{C}}

\newcommand{\Z}{\mathbb{Z}}
\newcommand{\N}{\mathbb{N}}

\newcommand{\CC}{\mathbb{C}}
\newcommand{\TT}{\mathbb{T}}
\newcommand{\ZZ}{\mathbb{Z}}
\newcommand{\NN}{\mathbb{N}}
\newcommand{\RR}{\mathbb{R}}

\newcommand{\FS}{\mathcal{F}^\star}

\newcommand{\BBB}{\mathcal{B}}

\newcommand{\supp}{\mathop{\rm supp}\nolimits}

\newcommand{\Ker}{\mathop{\rm Ker}\nolimits}

\DeclareMathOperator{\intt}{int}

\def\eop{\hfill$\Box$\par\bigskip}

\newcommand{\be}{\begin{equation}}
\newcommand{\ee}{\end{equation}}

\newtheorem{theorem}{Theorem}
\newtheorem{proposition}{Proposition}
\newtheorem{definition}{Definition}
\newtheorem{lemma}{Lemma}
\newtheorem{corollary}{Corollary}
\newtheorem{remark}{Remark}

\newtheorem{example}{Example}

\numberwithin{remark}{section}
\numberwithin{theorem}{section}
\numberwithin{proposition}{section}
\numberwithin{lemma}{section}
\numberwithin{corollary}{section}
\numberwithin{equation}{section}
\numberwithin{definition}{section}
\numberwithin{condition}{section}
\numberwithin{question}{section}
\numberwithin{conjecture}{section}
\numberwithin{problem}{section}
\numberwithin{example}{section}

\numberwithin{equation}{section}

\newcommand{\Tu}{{\mathcal T}}

\newcommand{\FF}{{\mathcal F}}
\newcommand{\DD}{{\mathcal D}}

\newcommand{\CCC}{{\mathcal C}}
\newcommand{\PP}{{\mathcal P}}

\newcommand{\GG}{{\mathcal G}}

\newcommand{\de}{\delta}
\newcommand{\ve}{\varepsilon}
\newcommand{\La}{\Lambda}
\newcommand{\al}{\alpha}

\newcommand{\De}{\Delta}

\newcommand{\vf}{\varphi}
\newcommand{\Omp}{\Omega_{+}}

\newcommand{\B}{{\mathcal{B}}_0}

\newcommand\FV{\mathcal{F}_{1}}
\newcommand\FC{\mathcal{F}_{c}}

\newcommand\CV{\mathcal{C}_1}
\newcommand\CCCC{\mathcal{C}_{c}}

\newcommand{\ft}[1]{\widehat{#1}}

\newcommand\EK{\mathcal{E}_{\kappa}}
\newcommand\DE{\mathcal{D}^{\mathcal{E}}}
\newcommand\DEK{\mathcal{D}^{{\mathcal E}_\kappa}}

\newcommand\CEK{\mathcal{C}^{{\mathcal E}_\kappa}}

\newcounter{rem}
\begin{document}

\title{Delsarte's Extremal Problem and Packing on \linebreak Locally Compact Abelian Groups}

\date{}

\author{Elena E. Berdysheva\thanks{The author would like to thank the Isaac Newton Institute for Mathematical Sciences, Cambridge, for support and hospitality during the programme Approximation, Sampling and Compression in Data Science where some work on this paper was undertaken. This work was supported by EPSRC grant no EP/R014604/1.  A part of the work was done when the author was visiting Alfr\'ed R\'enyi Institute of Mathematics of the Hungarian Academy of Sciences.  } \, and Szil\'ard Gy. R\'ev\'esz\thanks{Supported in part by Hungarian National Research, Development and Innovation Fund project \# K-119528 and DAAD-TEMPUS Cooperation Project ``Harmonic Analysis and Extremal Problems'' \# 308015.}}

\maketitle

\vskip1cm


\begin{abstract}

Let $G$ be a locally compact Abelian group, and let $\Omega_+$, $\Omega_-$ be two open sets in $G$. We investigate the constant
$
{\cal C}(\Omega_+,\Omega_-) = \sup{ \left\{ \int_G f: f \in {\cal F}(\Omega_+,\Omega_-) \right\} },
$
where ${\cal F}(\Omega_+,\Omega_-)$ is the class of positive definite functions $f$ on $G$ such that $f(0) = 1$, the positive part $f_+$ of $f$ is supported in $\Omega_+$, and its negative part $f_-$ is supported in $\Omega_-$. In the case when $\Omega_+ = \Omega_- =: \Omega$, the problem is exactly the so-called Tur\'an problem for the set $\Omega$. When $\Omega_- = G$, i.e., there is a restriction only on the set of positivity of $f$, we obtain the Delsarte problem. The Delsarte problem in $\Rd$ is the sharpest Fourier analytic tool to study packing density by translates of a given ``master copy'' set, which was studied first in connection with packing densities of Euclidean balls.

We give an upper estimate of the constant ${\cal C}(\Omega_+,\Omega_-)$  in the situation when the set $\Omega_+$ satisfies a certain packing type condition. This estimate is given in terms of the asymptotic uniform upper density of sets in locally compact Abelian groups.


\end{abstract}


\section{Introduction}\label{sec:Intro}

Let $G$ be a locally compact Abelian (LCA) group. We denote by ${\cal B}$ the class of Borel subsets of $G$, and by ${\cal B}_0$ the class of Borel subsets of $G$ whose closure is compact. We write $A_1 \Subset A_2$ if $A_1$ is a compact subset of $A_2$. For a set $A$, $\overline{A}$ denotes the closure of $A$ and $\intt{A}$ the interior of $A$, while $\chi_A$ stands for the characteristic function (indicator function) of $A$.

We denote by $m_G$ the Haar measure on $G$.  The convolution of two functions $f,g \in L^1(G)$ is defined by
$$
f * g := \int_G f(t) \, g(\cdot -t) \, dm_G(t).
$$
For a function $f : G \to \C$ we denote its converse function as $\widetilde{f}(x) := \overline{f(-x)}$. The support of a function $f$ is the closure of the set of all points where $f$ takes a non-zero value, i.e.,
$$
\supp{f} := \overline{ \{x : f(x) \ne 0 \}}.
$$

We will write $f \gg 0$ if $f$ is positive definite, i.e., if
\begin{equation}\label{posdefdef}
\sum_{n=1}^{N} \sum_{m=1}^N c_n \overline{c_m} f(x_n - x_m) \ge 0
\end{equation}
for all $N \in \N$, $x_1, \ldots, x_N \in G$ and $c_1,\ldots, c_N \in \C$.  For the basics on the harmonic analysis on LCA groups and for the facts about positive definite functions consult, e.g., the book of Rudin \cite{Rudin-book}.

For a real-valued function $f : G \to \R$, we use the standard notation
$$
f_+(x) := \max{\{ f(x), 0\}} \quad \text{and}  \quad f_-(x) := \max{\{ -f(x), 0\}};
$$
the functions $f_+$ and $f_-$ are the positive and the negative parts of $f$, respectively.
Let $\Omega_+$ and $\Omega_-$ be two open sets in $G$. We will consider real-valued positive definite functions $f$ on $G$ such that their positive and negative parts are supported in $\Omega_+$ and $\Omega_-$, respectively.  Depending on exact assumptions we put on the functions, we may consider different function classes. In this paper we mainly study the following function classes:
$$
\FV (\Omega_+, \Omega_-) :=
\left\{ f \in C(G)\cap L^1(G)  : f \gg 0, f(0) = 1, \supp{f_+} \subset \Omega_+, \supp{f_-} \subset \Omega_-  \right\},
$$
$$
\FC(\Omega_+, \Omega_-) :=
\left\{ f \in C(G) : f \gg 0, f(0) = 1 , \supp{f_+} \Subset \Omega_+, \supp{f_-} \Subset \Omega_- \right\}.
$$
So in the class $\FV(\Omega_+, \Omega_-)$, the sets $\supp{f_+}$ and $\supp{f_-}$ are closed sets that are not necessarily compact. Obviously,
\be \label{inclusion-c}
\FC(\Omega_+, \Omega_-)  \subset \FV(\Omega_+, \Omega_-).
\ee

The extremal problem we consider is to maximize the value of the integral of $f$ over
the function classes defined above. That is
we define the extremal constants
\begin{eqnarray} \label{C-problem}
\CV(\Omega_+, \Omega_-) & := &
\sup{ \left\{ \int_G f : f \in \FV (\Omega_+, \Omega_-) \right\} },
\nonumber \\
{\CCCC}(\Omega_+, \Omega_-) & := &
\sup{ \left\{ \int_G f : f \in \FC(\Omega_+, \Omega_-) \right\} }.
\end{eqnarray}
Note that for any meaningful interpretation of the extremal constants, the functions in our classes must be integrable, always. This explains the seemingly artificial restriction in the definition of $\FV$.

For an empty function class $\FF=\emptyset$ we interpret $\sup_{f\in \FF} \int_G f =0$. This is compatible with the easy fact that if $0\in \Omega_{+}$, then $\FC(\Omega_{+},\Omega_{-}) \supset \FC(\Omega_{+},\emptyset)\ne\emptyset$ and $\CCCC(\Omega_+, \Omega_-)>0$ (indeed, consider a properly normalized convolution $\chi_V * \widetilde{\chi_V}$, where $V \Subset \Omega_+$ with $V - V \subset \Omega_+$). On the other hand, for $0 \not\in \Omega_{+}$ we necessarily have $f(0)\le 0$, whence $f\gg 0$ implies $f\equiv 0$,  and thus $\FC(\Omega_+, \Omega_-)=\emptyset$.

As our first result we will show that the values defined above do not depend on a particular choice of the function class, and thus we can, and will denote the common value by ${\cal C}(\Omega_+, \Omega_-) $, or by ${\cal C}_G(\Omega_+, \Omega_-)$ if we want to emphasize the group we consider. This statement also means that we could study further function classes lying between $\FV(\Omega_+, \Omega_-)$ and $\FC(\Omega_+, \Omega_-)$; this would not change the value of the extremal problem. In fact, one can also extend (formally, as these would not actually increase the family of functions) the considered function class---we will continue to comment on it later in Section \ref{sec:EquivFormulations}.

For particular choices of $\Omega_-$, the problem  ${\cal C}(\Omega_+, \Omega_-) $ coincides with known extremal problems for positive definite functions. In the case when $\Omega_- = \Omega_+ =: \Omega$, it  is exactly the
so-called Tur\'an extremal problem\footnote{The problem became formulated and widely investigated after Tur\'an exposed to Stechkin \cite{stechkin:periodic} the corresponding question for intervals on the torus $\TT$. Although in the respective literature this extremal problem became widely known under Tur\'an's name, earlier, closely related results of Siegel \cite{Siegel}, Boas and Kac \cite{BK} and even Carath\'edory \cite{Cara} and Fej\'er \cite{Fej} surfaced in the paper \cite{Revesz-2011}.  This is why we term the extremal problem as the ``so-called''  Tur\'an problem. For a more detailed survey of the history of the problem and its background see \cite{Revesz-2011}. Thus the problem has been rediscovered several times on different occasions, including a recent paper~\cite{BianchiKelly}.}
\be \label{Turan-problem}
{\cal T}(\Omega)  := {\cal T}_c(\Omega) :=  \CCCC(\Omega,\Omega)
= \sup{ \left\{ \int_G f : f \in C(G), f \gg 0, f(0) = 1,  \supp{f} \Subset \Omega \right\} }.
\ee

Usually, in such context, one considers complex-valued functions, i.e., $f : G \to \C$. However, if $f$ is positive definite, then also $\Re{f}$ is positive definite, belongs to the same function class, $f(0) = \Re{f(0)}$, and $\int_G f = \int_G \Re{f}$. Thus, it is enough to consider only real-valued functions in problem (\ref{Turan-problem}). Also in problem (\ref{C-problem}) we could consider complex-valued functions $f : G \to \C$ with $(\Re{f})_+$,  $(\Re{f})_-$ supported in $\Omega_+$, $\Omega_-$, respectively. Since this does not change the value of the extremal problem, we restrict our consideration to the case of real-valued functions.

Since $f(x) = \overline{f(-x)}$ for positive definite functions $f$, the sets $\supp{f_\pm}$ are $0$-symmetric.  Thus, the condition $\supp{f_\pm} \subset \Omega_\pm$ implies also $\supp{f_\pm} \subset \Omega_\pm \cap (-\Omega_\pm)$, where the latter are already symmetric sets. Therefore we can assume without loss of generality that the sets $\Omega_\pm$ are symmetric.

If $\Omega_- = G$, we recover the Delsarte extremal problem
\begin{eqnarray*}
{\cal D}(\Omega_+)  := {\cal D}_c(\Omega_+)  := \CCCC(\Omega_+, G)
= \sup{ \left\{ \int_G f \right.} & : & f \in C(G), f \gg 0, f(0) = 1,\\ && {\left.\supp{f_+} \Subset \Omega_+, \supp{f_-} \Subset G \right\} },
\end{eqnarray*}
which, under the forthcoming Theorem \ref{equivalence-theorem}, is equal to
\begin{eqnarray*}
{\cal D}_1 (\Omega_+) & := & \CV (\Omega_+, G) \\
& = & \sup{ \left\{ \int_G f : f \in C(G) \cap L^1(G), f \gg 0, f(0) = 1, \supp{f_+} \subset \Omega_+  \right\} }.
\end{eqnarray*}
The term ``Delsarte's problem'' refers back to a classical paper of Delsarte~\cite{Del} where Delsarte used a completely analogously formulated extremal problem in case of discrete sets (codes) in terms of coefficients of Gegenbauer expansions (in place of Fourier transforms), see also~\cite[Theorem 4.3]{Del3}. Following Delsarte, these problems were used to obtain estimates for densities of sphere packings, kissing numbers, cardinalities of spherical codes, etc.; see, e.g. \cite{AreBab97, AreBab00, Boy, cohn:packings, gorbachev:sphere, gorbachev:ball, KabLev, Kuklin, Lev, Musin, Viaz}.

In case of sphere packings in $\R^d$, exactly the above Delsarte extremal problem of finding ${\cal D}(B)={\cal C}(B,\R^d)$ occurs \cite{gorbachev:sphere, cohn:packings, Viaz}, where $B:=\{ {\bf x} \in \RR^d~:~|{\bf x}|  <  1\}$  denotes the unit ball of $\R^d$,  apart from choosing appropriate function classes varying from author to author but essentially equivalent to the classes $\FC(B, \R^d)$ and $\FV(B, \R^d)$. As mentioned above, we will show that the choice of the particular function class---at least as long as the ball is considered---is immaterial.

Obviously, if $\Omega_+ \subset \widetilde\Omega_+$ and $\Omega_- \subset \widetilde\Omega_-$, then ${\cal C}(\Omega_+, \Omega_-)  \le {\cal C}(\widetilde\Omega_+, \widetilde\Omega_-)$. In particular, ${\cal T}(\Omega) \le {\cal D}(\Omega)$. It can happen that  ${\cal T}(\Omega) = {\cal D}(\Omega)$. This is the case when $G = {\mathbb T}$ is the one-dimensional torus and $\Omega = [-h,h]$. Both problems ${\cal T}([-h,h])$ and ${\cal D}([-h,h])$ on the torus were solved in a series of papers \cite{IvanovRudomazina:rational_h, IvanovGorbachevRudomazina:IMM, Ivanov:all_h}, see also references therein.  The equality ${\cal T}(\Omega) = {\cal D}(\Omega)$ also holds in the general setting when $\Omega_+$ is a difference set of a strict tile (see Proposition~\ref{tile-theorem}). 

However, the inequality  ${\cal T}(\Omega) \le {\cal D}(\Omega)$  can be strict. This is, e.g., the case for $\Omega$ being the Euclidean ball $B$ in $\R^d$.
Indeed,  it is known for long that ${\Tu }(B)= 2^{-d} |B| =2^{-d} \omega_d $ with $\omega_d=|B|$ denoting the volume of the unit ball in $\R^d$, see \cite{Siegel, gorbachev:ball, kolountzakis:turan, BianchiKelly}; as metioned in~\cite{BianchiKelly}, the Tur\'an problem for the ball was also implicitly solved in~\cite{HoltVaaler}. On the other hand, for $d=8$, for example, ${\cal D}(B) = 2^{-4} = 0.0625$ as has been shown by Viazovska~\cite{Viaz}, which is considerably larger than ${\cal T}(B) = 0.015854...$. This is not just a numerical difference but a very crucial one because the Delsarte bound unlike the Tur\'an bound turned to be exact in this case regarding the density of sphere packing.

The first attempt to use such Fourier analytic extremal problems to establish bounds for packing densities was worked out by Siegel~\cite{Siegel} using ${\cal T}(\Omega)$ but later it turned out that the Delsarte extremal problem can give sharper bounds in most of the situations.

Once packing density is mentioned, it is a point that construction of the notion of appropriate densities is not always trivial. To obtain sharpest bounds, one is looking for the largest reasonable variants of densities, which are well-known (and are called asymptotic uniform upper densities or Banach densities) in classical cases like e.g. $\RR^d$ or $\ZZ^d$, but were not constructed to general LCA groups until recently. We will explain the notion of these densities in Section~6, see in particular Definition \ref{def-auud}. For a discrete set $\Lambda \subset G$ we will denote its asymptotic uniform upper density as $\overline{D}^{\#}(\Lambda)$, see \eqref{dens-discr-def}.

The main aim of the paper is to study the behavior of the constant ${\cal C}(\Omega_+,\Omega_-)$ in the case when the positivity set $\Omega_+$ possesses some structural properties like packing and tiling. While the connection between the packing densities and the Tur\'an problem were explored in \cite{KR-2006, Revesz-2011}, our point here is to further the analysis  to the connection of packing type density questions and the Delsarte problem. Furthermore, our analysis reveals that Delsarte type bounds can be applied under more general hypothesis than mere packing.
Following  \cite{KR-2006, Revesz-2011}, we say  that a set $W \in {\cal B}_0$ satisfies a \emph{generalized strict packing type condition} with the translation set  $\Lambda \subset G$ if
\begin{equation}\label{eq:generalpackingtype-strict}
(\Lambda - \Lambda) \cap W \subseteq \{0\},
\end{equation}
see Definition~\ref{def:genpack}. In particular, we say that $H\in {\cal B}_0$ packs $G$ with the translation set $\Lambda \subset G$ \emph{in the strict sense} if
$$
\sum_{\lambda \in \Lambda} \chi_H(x - \lambda) \le 1 \qquad \text{for all} \ x \in G,
$$
which is equivalent to
\begin{equation}\label{stricpacking}
(\Lambda - \Lambda) \cap (H - H) = \{0\}.
\end{equation}
This reveals how \eqref{eq:generalpackingtype-strict} generalizes \eqref{stricpacking} when $W$ is not necessarily a difference set.

The main result of the paper is the following statement.

\begin{theorem} \label{main-theorem}
Let $G$ be a LCA group. Suppose $\Omega_+ \subset G$ is an open $0$-symmetric neighborhood of $0$ satisfying the strict packing type condition with a translation set $\Lambda \subset G$:
\begin{equation}\label{eq:generalpackingtype-strict-Omega}
\Omega_+ \cap (\Lambda - \Lambda) = \{0\}.
\end{equation}
Then
\begin{equation} \label{main-theorem-estimate}
{\cal D}(\Omega_+) \le \frac{1}{\overline{D}^{\#}(\Lambda)}.
\end{equation}

\end{theorem}

\begin{corollary}
Let $G$ be a LCA group. Suppose $\Omega_+ \subset G$ is an open $0$-symmetric neighborhood of $0$ satisfying the strict packing type condition \eqref{eq:generalpackingtype-strict-Omega} with a translation set $\Lambda \subset G$. Let $\Omega_-$ be any open, $0$-symmetric set. Then
$$
{\cal C}(\Omega_+, \Omega_-) \le \frac{1}{\overline{D}^{\#}(\Lambda)}.
$$

\end{corollary}

The paper is organized as follows. In Section~\ref{sec:EquivFormulations} we will show that the value of the extremal problem does not depend on the particular function class $\FC (\Omega_+, \Omega_-) $ or $\FV (\Omega_+, \Omega_-)$. Moreover, we can even (formally) extend the class $\FV (\Omega_+, \Omega_-)$ without changing the value. In Section~\ref{sec:account} we concentrate on the particular case $G = \R^d$, mainly considering $\Omega$ being the Euclidean ball $B$, and review the function classes used by other authors like Gorbachev, Cohn and Elkies, Viazovska, with the result that the extremal value in all these cases is the same as in our setting. The key result of this section is a statement of Gorbachev that ${\cal D}(B)$ coincides with the value $\DD^\GG_0(B) := \sup \left\{ \int_{\RR^d} f~:~ f \in \GG_0(B)\right\}$   of the Delsarte problem for the class of functions
$$
\GG_0(B):=\left\{ f \in C(\RR^d) \cap L^1(\RR^d) :  f \gg 0, f(0) = 1, \supp{f_+} \subset B, \supp \widehat{f} \Subset \RR^d \right\}.
$$
The proof of this fact was not recorded earlier and is given here in a friendly agreement with Dmitry Gorbachev.

In Section 4 we study the behavior of the extremal problem ${\cal C}(\Omega_+, \Omega_-)$ under homomorphisms. In Section 5 we discuss packing in the strict sense and the generalized strict packing type condition.
In a short Section 6 we explain the notion of the asymptotic uniform upper density on LCA groups. In the final Section~7 we prove Theorem~\ref{main-theorem}.

\bigskip

\noindent {\bf Acknowledgments.} 
We thank Dmitry Gorbachev for numerous conversations and suggestions and for providing us Proposition~\ref{prop:Gorbachev}, as well as its proof including the auxiliary statements and references leading to this result. We also thank Valerii Ivanov, Vilmos Komornik, Marcello Lucia and Gerg\H{o} Nemes for useful discussions and references. We thank the anonymous referee, too, for directing our attention to further relevant literature.


\section{Equivalence of the extremal problems in various function classes}\label{sec:EquivFormulations}

In this section we will prove that the value of the extremal problem in (\ref{C-problem}) does not depend on the particular choice of the function class as given in the above definitions. Although this may seem a mere technicality, it requires a proof anyway. Moreover, note that we will encounter variants (classes of functions with compactly supported Fourier transform in Section \ref{sec:account}), where this equivalence is only known in rather special cases. Therefore, one has to be careful with underestimating these ``mere technicalities''.

\begin{theorem}\label{equivalence-theorem}
If $\Omega_+$ and $\Omega_-$ are open, $0$-symmetric subsets of a LCA group $G$,
then
$$
\CV(\Omega_+, \Omega_-) = \CCCC(\Omega_+, \Omega_-).
$$

\end{theorem}

The corresponding statement for the Tur\'an problem (\ref{Turan-problem}) was proved in \cite{KR-2006} in a somewhat different variant. Our proof is analogous.

\begin{proof}
We only need to consider the case $0 \in \Omega_+$, for if $0 \not\in \Omega_+$ then  both values above are zero. Also, inclusion (\ref{inclusion-c}) implies $\CCCC(\Omega_+, \Omega_-) \le \CV(\Omega_+, \Omega_-)$, so we need to show the converse inequality only.

Let $\ve > 0$.  There is a function $f \in {\cal F}_1(\Omega_+, \Omega_-)$ such that $\int_G f \ge {\cal C}_1(\Omega_+, \Omega_-) - \ve$. Since $f \in L^1(G)$, there is a compact set $C \Subset G$ such that $\int_{G \setminus C} |f| < \ve$.  Then, in particular, $\int_C f \ge {\cal C}_1(\Omega_+, \Omega_-) - 2\ve$.

Next we will use the well-known fact that the constant one function ${\bf 1}$ can be approximated locally uniformly by continuous positive definite functions of compact support. As we will need this several times in our paper, let us formulate it as a lemma.

\begin{lemma}[{\bf Approximation of unity lemma}]\label{l:Kg} Let $C\Subset G$ and  $\varepsilon>0$ be arbitrary. Then there exists $k\gg 0$, $k\in C_{c}(G)$ (so continuous with compact support)
and $0 \le k \le 1$,  such that $k|_C \ge 1-\ve$ and $\|k\|_\infty=k(0)= 1$.
\end{lemma}

For the proof, see e.g. \cite[2.6.8. Theorem]{Rudin-book} (where, however, the formulation is somewhat different) or, for precisely this form, \cite[Lemma 2]{KR-2006} or \cite[Lemma~5]{GR}. See also \cite[Problem 5]{God}.

Now consider $g := fk$.  Obviously, $g \in C(G)$ and $g(0) = 1$. Moreover, $g \gg 0$ as a product of positive definite functions\footnote{This is a nontrivial fact, which follows from the Schur Product Theorem: if the matrices $A = [a_{jk}]_{j=1,...,n}^{k=1,...,n}, B = [b_{jk}]_{j=1,...,n}^{k=1,...,n} \in \CC^{n\times n}$ are both positive definite, then so is their entry-wise (Schur- or Hadamard-) product $[a_{jk}b_{jk}]_{j=1,...,n}^{k=1,...,n}$, too. See also \cite[\S 85, Theorem 2]{Halmos}. The statement can be found e.g. in \cite[(32.8) (d)]{HewittRossII} and \cite[(32.9) Theorem]{HewittRossII}, see also \cite[Lemma 12(v)]{KrRe2}, \cite[Lemma 2.6.1 (iii)]{Krenedits}.}.
Since $\supp{g_\pm} \subset \supp{f_\pm} \cap \supp{k}$, i.e., a closed subset of the compact set $\supp k$, it is compact, too: $\supp{g_\pm} \Subset \Omega_\pm$. Thus, $g \in  \FC(\Omega_+, \Omega_-)$. Clearly,
$$
\int_G g \ge \int_{C} f - \ve \int_{C} |f| - \int_{G \setminus C} |f|
\ge \left( {\cal C}_1(\Omega_+, \Omega_-) - 2\ve \right) - \ve\|f\|_{L^1(G)} - \ve.
$$
Since $\ve > 0$ can be taken arbitrarily,  $\CCCC(\Omega_+, \Omega_-) \ge \CV(\Omega_+, \Omega_-)$ follows.
 \eop \end{proof}

Next we would like to comment on the notion of positive definiteness. Above we defined a positive definite function through \eqref{posdefdef}, but quite often positive definiteness of functions is understood with various  different meanings, which in many cases are equivalent from the point of view of the analyzed questions, but sometimes exhibit differences, too. There are two further major ways of defining (some kind of) positive definite functions differently, which we briefly mention here.

First, \emph{an almost everywhere defined measurable ``function''} (in precise terms, the respective equivalence class of functions) is called a \emph{function of positive type}, if it is locally Haar-integrable and if for ``test functions'' from $C_c(G)$ ($C_c(G)$ denoting the family of continuous functions of compact support) it holds
\begin{eqnarray} \label{typecond}
&\int_{G}f ~(\tilde{u} * u) ~d m_G \geq 0 \quad \text{ for all } u\in C_c(G) \quad \textrm{or, equivalently,} \nonumber\\
& (\widetilde{u} * u * f)(0) \ge 0 \quad \text{ for all } u\in C_c(G) .
\end{eqnarray}

This definition follows Godement \cite{God}, but adopts the later terminology of e.g. Folland \cite{Folland}.
Note the distinction between the classes of positive definite \emph{functions}, defined \emph{finitely everywhere} and satisfying \eqref{posdefdef}, and \emph{``functions'' of positive type}, defined only a.e. in accordance with \eqref{typecond}. As a matter of fact, one can define functions of positive type \emph{with respect to a given class of functions}, playing the role of $C_c(G)$ above---it seems that this idea was first analyzed by Cooper \cite{Cooper}.

Also, positive definiteness is sometimes understood simply as \emph{nonnegativity of the Fourier transform}, so that positive definite functions are tacitly assumed to be functions from the inverse image with respect to the Fourier transform of (some family of) nonnegative functions (or measures, or distributions). This is the working assumption e.g. in the paper of Logan \cite{Log}. However, there can be a huge ambiguity here with respect to classes of functions, classical, $L^2$, or distributional Fourier transforms etc. Therefore, if such a situation is encountered, we prefer to make it clear that the function in question has nonnegative Fourier transform (and in what sense).

Each positive definite function in the sense of  \eqref{posdefdef} is trivially bounded by its value at zero, i.e., $|f(x)| \le f(0)$ for all $x \in G$.

A positive definite function needs to be neither continuous nor even measurable. However, each measurable positive definite function coincides locally a.e.\footnote{A property is true locally almost everywhere when the exceptional set is a locally null set, i.e., its intersection with any Borel set of finite measure is a Borel set of zero measure, cf.~\cite{Folland}. This concept is needed when one is dealing with Haar measures that are not $\sigma$-finite.} with a continuous positive definite function, see \cite[(32.12) Theorem]{HewittRossII}. Also, for a bounded continuous function, being a positive definite function is equivalent to being a function of positive type \cite[3.35 Proposition]{Folland}.

By \cite[3.21 Corollary and 3.35 Proposition]{Folland}, every measurable and bounded function of positive type agrees locally a.e. with a continuous positive definite function. Further, in the extremal quantities, what we are to investigate, our goal function (to be maximized) is only an integral (over the full group G). Whence if we assume boundedness of a function of positive type, we can as well restrict considerations to continuous positive definite functions. The key lies slightly deeper when possibly unbounded functions of positive type are concerned, but even for those \emph{local boundedness at zero}---which is essentially the normalization of our extremal problem, i.e. the condition $f(0)=1$---suffices. This, however, requires (a strong version of) the celebrated Gelfand-Raikov Theorem, too, see e.g. \cite[Theoreme 3]{God} or \cite[Theorem 7]{GR}. If, however, we use the Gelfand-Raikov Theorem, then even \emph{measures of positive type} can be handled the same way.

\begin{definition}\label{def:postypemeasureextremal} If $\mu$ is a finite regular Borel measure on $G$, then we define the ``local essential supremum of the effect of $\mu$'' as \begin{equation}\label{eq:cmudef}
c(\mu):=\inf_{0\in U \,{\rm open}} \quad \sup_{0\le u\in C_c(G), ~\supp u \Subset U, ~\int_G u=1} \quad \int_G u * \widetilde{u} ~d\mu.
\end{equation}
\end{definition}

Note that $c(\mu)$ is the proper generalization of the function value of a continuous function at 0. Indeed, if $d\mu =f dm_G$ (at least in some neighborhood of $0$), with $f$ continuous at $0$, then one can easily derive that $c(\mu)=f(0)$. Also note that $c(\mu)$ is not necessarily finite: but in fact, \emph{in case $\mu$ is of positive type}, the property that $c(\mu)<\infty$ is \emph{equivalent} to the assertion that $\mu$ is absolutely continuous and its density function (which is determined only a.e. by the measure) can be taken a continuous function $f\gg 0$ with $f(0)=c(\mu)$. For the proof of this converse statement one needs to use (a strong version of) the Gelfand-Raikov Theorem---for details see e.g. \cite[Theorem 7]{GR}.

In view of the above, we can as well consider the more general ``function class''---class of measures---normalized by assuming $c(\mu)=1$. Then we could write
$$
{\cal F}^\star(\Omega_+, \Omega_-) :=
\left\{ \mu ~\textrm{of positive type} : c(\mu) = 1,  \supp{\mu_+} \subset \Omega_+, \supp{\mu_-} \subset \Omega_-  \right\}.
$$
Correspondingly, we can define the respective extremal value
$$
{\cal C}^\star(\Omega_+, \Omega_-):=\sup \left\{\int_Gd\mu~:~\mu\in \FS(\Omega_+, \Omega_-) \right\}.
$$
Note that, once again, for any meaningful interpretation of the extremal problem (the goal function itself), the measure $\mu$ must be of finite total variation.

Then, according to the above, we would conclude that in fact the measures in \linebreak $\FS(\Omega_+, \Omega_-)$ are all locally absolutely continuous, and the density functions of the measures appearing in $\FS(\Omega_+, \Omega_-)$ can be identified with the continuous positive definite functions from $\FV(\Omega_+, \Omega_-)$. Therefore, the extension of the extremal problem from $\FV$ and $\CV$ to $\FS$ and ${\cal C}^\star$ is only formal, with no actual change.

\section{An account of related work on Delsarte type extremal constants on various function classes} \label{sec:account}

This section is only useful to see that our general investigations do indeed cover the actual applications of the Delsarte method in e.g. sphere packing. It should have been obvious, but in fact it is not because various authors in various papers used different formulations and in particular different function classes for their formulations of the analogous questions. Of course, the variety is basically justified by the general equivalence of most of these choices---however, these need to be proven. Moreover, there did occur non-equivalent versions, too, which, again need to be clarified. The reader should not expect any attracting details and may want to fully skip this section, but we felt it an obligation to tediously clarify these connections, however boring.

First, to better facilitate our discussion to existing literature, we extend the definition of function classes $\FC(X,Y)$ and $\FF_1(X,Y)$---as well as the respective Delsarte-type constants \eqref{C-problem}---to arbitrary Borel measurable and symmetric sets $X,Y \in\BBB$. In the large, whenever these classes and constants appear in the paper, we still refrain to open sets---the extended definition will be in effect exclusively in this section for the sake of the comparisons we want to explain.

As a first observation, we need to mention that in view of the obvious fact that the support of a continuous function is closed, by definition of $\FF_{1}(X,Y)$ we have $\FF_{1}(X,Y) = \bigcup_{E\subset X, F\subset  Y} \FF_{1}(E,F)$, where here the union runs on \emph{closed} sets $E, F$ contained in $X, Y$, respectively. Therefore, we have $\CCC_{1}(X,Y)=\sup \{ \CCC_{1}(E,F)~:~ E\subset X, F\subset Y, ~~E,F\,{\rm closed} \}$.

We will need a little more, namely, that the above limit or supremum relation holds true with compact sets (in place of closed ones) as well. The easiest is to prove this by means of Lemma \ref{l:Kg} above. The argument of the proof of Theorem \ref{equivalence-theorem} goes through mutatis mutandis. So, for general sets, too,
$$
\CCC_{1}(X,Y)=\sup \{ \CCCC(E,F)~:~ E\Subset X, F\Subset Y\},
$$
and, consequently,
\begin{equation}
\label{equiv-arb-set}
\CCC_{1}(X,Y) = \CCCC(X,Y) .
\end{equation}

In connection to this, however, let us warn the reader that for open sets it is not automatic that the extremal constant of the closures would match to that of the original open sets---see the counterexample of Theorem 7 of \cite{KR-2006}. Therefore, we will be prudently restricting ourselves in stating (and even proving) that nevertheless, at least in the classical Euclidean spaces and convex sets, the equivalence of closed or open copies holds true.

\begin{proposition} Let $X,Y \subset \RR^d$ be centrally symmetric (with respect to the origin) convex sets with the origin in the interior of $X$. Then $\CCC(\intt X, \intt Y)=\CCC(\overline{X}, \overline{Y})$.
\end{proposition}
An analogous statement holds in $\TT^d$ for small enough sets $X,Y$ such that not even the closures overlap with themselves---we leave the details to the reader.

\begin{proof}
We will use the fact that a convex set $X\subset \RR^d$ \emph{with
nonempty interior} is \emph{fat}, meaning that
$\overline{X}=\overline{\intt X}$. As $0\in \intt X$, we thus have that
with an appropriately small $\de>0$ the closed ball $\de B \subset \intt X$, whence by
convexity $r\overline{X}+\de(1-r) B \subset \intt X$ and $r\overline{X}
\subset \intt X$ for $0\le r <1$.

In case $X$ is unbounded, it is easy to see that both extremal constants become $+\infty$. Indeed, taking $Z:= \lambda X \cap RB$ with any $\lambda <1/2$ and $R>0$, the function $f:=\frac{1}{|Z|}\chi_{Z} * \chi_{Z}$ belongs to both function classes and has as large integral as $|Z|$, which tends\footnote{Note that the unbounded convex set with $\de B\subset X$ contains a cone of arbitrarily large height and base $\delta B_{d-1}$, whence has infinite volume.} to $|\lambda X| = \infty$ with $R\to \infty$.

So let us assume that $X$ is bounded, which also implies that both extremal constants are finite: for any positive definite function $f$ with $\supp f_+\subset X$ we have $0\le \int_{\RR^d} f \le \int_{\RR^d} f_{+} = \int_{\supp f_{+}} f \le \int_{\supp f_{+}} f(0) \le |X| \cdot 1 = |X| < \infty$.

First consider the case when also $0 \in \intt Y$ and hence also $Y$ is fat.

Obviously $\CCC(\intt X, \intt Y)\le \CCC(\overline{X}, \overline{Y})$.
Let us take any $\ve>0$ and $f\in \FC(\overline{X}, \overline{Y})$ with
$\int_{\RR^d} f \ge \CCC(\overline{X}, \overline{Y})-\ve$. Take any $R>1$,
put $r:=1/R<1$, and consider $g(x):=f(Rx)$. Obviously $\supp g_{\pm} = r
\supp f_{\pm} \Subset r \overline{X}, r \overline{Y} \subset \intt X,
\intt Y$, respectively. Also, $g \gg 0$, $g(0)=f(0)=1$ and $g \in
C_c(\RR^d)$, so $g \in \FC(\intt X, \intt Y)$ and, moreover, $\int_{\RR^d}
g(x) dx = \int_{\RR^d} f(Rx) dx = \int_{\RR^d}  f(y) r^d dy  \ge r^d\left(
\CCC(\overline{X}, \overline{Y})-\ve \right)$ furnishing $\CCC(\intt X,
\intt Y)\ge r^d (\CCC(\overline{X}, \overline{Y})-\ve)$. With this proven,
we can allow $R\to 1+0$, that is $r\to 1-0$ and $\ve \to 0+$, furnishing
the result.

If $0 \not\in \intt Y$, then $Y$ being centrally symmetric is actually lying in a hyperplane. So it remains to deal with the ``strange'' case when $\intt Y = \emptyset $ and $|Y|=0$.  As above, take any $\ve>0$ and $f\in \FC(\overline{X}, \overline{Y})$ with $\int_{\RR^d} f \ge
\CCC(\overline{X}, \overline{Y})-\ve$. In principle the function $f$ may attain negative values (namely, on $\overline{Y}$), but as $\overline{Y}$ is on the boundary of $\RR^d\setminus \overline{Y}$, we conclude by continuity that $f$ is nonnegative all over $\RR^d$.

Then the same construction as above gives a function $g(x):=f(Rx)$
belonging to $\FC(\intt X,\emptyset)$, and that yields, as above, the
result.
 \eop \end{proof}

In \cite{gorbachev:ball}, Gorbachev studies Tur\'an's problem in the following setup. For a given centrally symmetric body $\Omega$, he maximizes $g(0)$ in the class of continuous even functions~$g : \R^d \to \R$ satisfying
\begin{enumerate}
\item $g(y) = \int_{\Omega} \widehat{g} (x) e^{2\pi iyx} dx$,
\item $g(y) \ge 0$ for all $y \in \R^d$,
\item $\widehat{g}(0) = \int_{\R^d} g(y) dy = 1$.
\end{enumerate}

In this setup, $\widehat{g}$ corresponds to our function $f$. Here $f = \widehat{g}$ is continuous (as the Fourier transform of an $L^1$-function) and positive definite (as the Fourier transform of a nonnegative function). By (iii) we also have  $f(0)=1$, and thus  $f$ lies in our class ${\cal F}_1(\Omega, \emptyset)$. The only difference is that in general we consider open sets $\Omega$ while  in all examples in \cite{gorbachev:ball}  $\Omega$ is closed.

\bigskip

In \cite{gorbachev:sphere}, Gorbachev studies the Delsarte problem for the class of positive definite, continuous, real functions $f \in L^1(\R^d)$ with their Fourier transforms vanishing outside of the Euclidean ball $rB$ of a given radius $r=r_d$. These functions are entire functions of the spherical exponential type $r$. Gorbachev uses the analog of the Delsarte problem on this class to derive an upper estimate on the density of any possible spherical packing, and then he gives the exact solution of the Delsarte problem on this class for the concrete radius that depends on the dimension.

The difference in this setting is that restriction on the function class is imposed on the Fourier transform side, namely  $\supp \widehat{f} \Subset rB$. That is,  the class of functions in \cite{gorbachev:sphere}   is $\GG(B ; rB)$, where in general
\begin{equation}\label{eq:Gorbachevclass}
\GG(W, Q) :=
\left\{ f \in C(\RR^d)  \cap L^1(\RR^d)  :
f \gg 0, f(0) = 1,  \supp{f_+} \subset W,  \supp \widehat{f} \subset Q \right\}.
\end{equation}
Note that we did not mention $f_{-}$---and indeed, it is only taken that $\supp f_{-} \subset \RR^d$, that is, no restriction. Once a strong restriction is applied on the Fourier transform side---e.g. if $Q$ is bounded and hence $\supp \widehat{f} $ is compact, meaning that $f$ is an entire function---it is no longer possible to restrict, e.g, to $\supp{f_-} \Subset \RR^d$, as $f_{+}$ already supported compactly (say when $W$ is also bounded, like in the central case when $W=B$) would then imply $\supp f \Subset \RR^d$, which is not possible for entire functions (if we assume $f(0)=1$, i.e. $f \not\equiv 0$, too). For a similar comment see the end of page 699 in \cite{cohn:packings}.

Gorbachev proves the estimate $\Delta_d \le \frac{\omega_d}{2^d \DD^\GG(B,rB)}$ ($\forall r>0$) for the  maximum packing density\footnote{$\Delta_d$ is the maximal possible density of a packing of $\R^d$ by spheres. We postpone going into details on this and related notions of density until our Section~\ref{sec:auud}. For a discussion on densities of spherical packings see, e.g.,~\cite{cohn:packings}, (in particular Appendix A of ~\cite{cohn:packings}) and references therein. }
$\Delta_d$ of $\RR^d$ by unit balls $B$ using Poisson summation, and then computes the exact value of $\DD^\GG(B,r_dB)$, where his extremal constant  is
\begin{equation}\label{eq:Gorbiconst}
\DD^\GG(W,Q):= \sup \left\{ \int_{\RR^d} f~:~ f \in \GG(W,Q)\right\}.
\end{equation}

Obviously, as is remarked in \cite{cohn:packings}, one can then write $\Delta_d  \le \lim_{r\to \infty} \frac{\omega_d}{2^d \DD^\GG(B,rB)} = \frac{\omega_d}{2^d \DD^\GG_0(B)}$, where in general $\DD^\GG_0(W) := \sup \left\{ \int_{\RR^d} f~:~ f \in \GG_0(W)\right\}$ with
$$
\GG_0(W):=\left\{ f \in C(\RR^d) \cap L^1(\RR^d) :  f \gg 0, f(0) = 1, \supp{f_+} \subset W, \supp \widehat{f} \Subset \RR^d \right\}.
$$
However,  the  approach in \cite{gorbachev:sphere} found the exact constant $\DD^\GG(B,rB)$ only for the special value $r_d =  2 q_{d/2}$, where $q_{d/2}$ is the first positive root of the Bessel function $J_{d/2}$ of the first kind, which is of course smaller than the actual value of $\DD^\GG_0(B)$. As a result, his upper estimation of the packing density $\Delta_d$ exceeded $\frac{\omega_d}{2^d \DD^\GG_0(B)}$, whence could not be sharp.

Let  us now record that our definition of the Delsarte constant \emph{in case of the unit ball} is equivalent to $\DD^\GG_0(B)$, the limiting case of $\DD^\GG(B,rB)$ (as it was already stated in \cite{gorbachev:sphere} and is also mentioned by Cohn and Elkies on page 694 of \cite{cohn:packings} with respect to their definition, see below).

It is interesting to note that a question of the same type had been risen in a recent paper~\cite{GOR:JGeomAnal2021} by Gon\c{c}alves, Oliveira e Silva and Ramos for closely related extremal values ${\mathbb A}_+(d)$ and ${\mathbb A}_-(d)$, $d$ being the dimension, in problems they name the sign uncertainty principles. In particular, the first equality ${\mathbb A}_s(d)={\mathbb A}_s^{\mathcal B}(d)$ (with the sign $s \in\{-,+\}$) in Conjecture~1 in~\cite{GOR:JGeomAnal2021}---proven there for ${\mathbb A}_+(1)$ as Theorem 1---is very similar to Gorbachev's result in Proposition~\ref{prop:Gorbachev} below. These can be also considered as an example of work in the spirit of our Section~\ref{sec:EquivFormulations}: the value of an extremal problem does not depend on a choice of a particular function class.

The quantity denoted as ${\mathbb A}_{LP}(d)$ in \cite{GOR:NewSign} is equivalent to the Delsarte problem for the ball: more precisely, ${\mathbb A}_{LP}(d)=1/\DD(B)^{1/d}$. It has been conjectured for long that the problem ${\mathbb A}_-(d)$ is also equivalent to these problems. See Conjecture 7.2 in \cite{cohn:packings}, recalled by Gon\c{c}alves, Oliveira e Silva and Ramos in the equivalent form ${\mathbb A}_-(d)={\mathbb A}_{LP}(d)$ as \cite[Conjecture~1.5]{GOR:NewSign}; the equivalence  of the two formulations proven already in \cite{CohnGon:InvMath2019}.

The existence of extremal functions for the problems ${\mathbb A}_s(d)$, $s \in \{+,-\}$, was shown in~\cite{GOS:JMAA2017, CohnGon:InvMath2019}, compare to our Proposition~\ref{pr:existence-extremal} and Remark~\ref{rem:LCA-existence} below.

Now let us come back to the announced statement about the equivalence of the problems $\DD^\GG_0(B)$ and $\DD(B)$. In fact, this is not trivial at all; below we record the proof, kindly provided to us by Dmitry Gorbachev in personal communication. Note that this proof was not written down previously. We do not see a proof for the analogous fact in a more general setting---radial symmetry is important in the construction, so already for convex bodies $W \subset \RR^d$ the situation is unclear.

\begin{proposition}[Gorbachev]\label{prop:Gorbachev} For the ball $B\subset\RR^d$, we have $\DD^\GG_0(B)=\DD(B)$.
\end{proposition}

\begin{proof} (Gorbachev)
Since for any symmetric Borel set $W \subset \RR^d$ obviously we have $\GG_0(W) \subset \FV (W, \RR^d) $, we find $\DD^\GG_0(W) \le {\mathcal D}_1(W)$, and \eqref{equiv-arb-set} yields the estimate  $ {\mathcal D}(W) \ge \DD^\GG_0(W)$.

The nontrivial part of the argument is thus to show that $\DD(B) \le  \DD^\GG_0(B)$ for the ball~$B$.
The proof depends on two constructions by Yudin and by Gorbachev, respectively; we describe them in the following two lemmata. For the reader's convenience, we also provide full proofs.

Let $J_{\alpha}$ denote the Bessel function of the first kind, and let
$
j_{\alpha}(t) = \Gamma(\alpha +1) \left( \frac{2}{t} \right)^{\alpha} J_{\alpha}(t)
$
be the normalized Bessel function with the property $j_{\alpha}(0) = 1$. We denote by $q_{\alpha}$ the first positive zero of~$j_{\alpha}$.

\begin{lemma}[Yudin]\label{lem-Y}
The functions
\begin{equation}\label{Yudin-function}
Y_d(t) :=  \frac{j_{\frac{d}{2}-1}^{2}(t)}{1 - t^2 / q_{\frac{d}{2}-1}^{2}}, \quad t \in [0,\infty),
\end{equation}
and
$$
y_d(x) := Y_d(|x|), \quad x \in \R^d,
$$
have the following properties:
\begin{itemize}
\item[{\rm (i)}] $Y_d(t) \ge 0 ~{\rm for} ~ 0 \le t \le q_{\frac{d}{2} -1}, \quad Y_d(t) \le 0 ~{\rm for} ~   t \ge q_{\frac{d}{2} -1}$,
\item[{\rm (ii)}] $\widehat{y}_d \ge 0, \quad \supp{ \widehat{y}_d } \Subset 2B, \quad \widehat{y}_d(0) = 0$.
\end{itemize}

\end{lemma}

\begin{proof}
The construction presented in the lemma and the proof of the above properties were given by Yudin in~\cite{Yudin} in a much more general case. We give here a proof adopted to our particular situation.

We use the known fact that the function $u(x) = j_{\frac{d}{2}-1}\left( q_{\frac{d}{2}-1}|x| \right)$ is the ``first Dirichlet eigenfunction of the ball $B$'', (with the first eigenvalue $q_{\frac{d}{2}-1}^2$), i.e.
\begin{align*}
\Delta u(x) & =  - q_{\frac{d}{2}-1}^2  u(x), \quad x \in B, \\
u \big|_{S} & =   0,
\end{align*}
where $S := \partial B$ denotes the sphere of radius $1$ in $\R^d$. In fact, looking for radial solutions of the above Dirichlet eigenvalue problem (with arbitrary eigenvalue) it is natural to re-write the Laplace operator in radial coordinates, which then leads to a scaled version of the Bessel differential equation; thus radial solutions are of the form $C j_{\frac{d}{2}-1}\left( \lambda |x| \right)$, where here $\lambda$ must be a zero of the Bessel function $j_{\frac{d}{2}-1}$ in view of the Dirichlet boundary condition $u|_{S} =   0$. This argument finds only radial solutions and thus leaves room for hypothetical non-symmetric solutions, too, but it is clear that any solution leads to another one, with the same eigenvalue, but now symmetric, by spherical averaging.
Take
$$
\varphi(x) :=\begin{cases}
u(x), \quad &\text{if}~~ x\in B, \\
0, \quad &\text{if}~~ x\not\in B.
\end{cases}
$$
With $q_{\frac{d}{2}-1}$ being the first zero of $j_{\frac{d}{2}-1}$, the function $\varphi$ is obviously nonnegative (and not identically zero as $\varphi(0)=1$). However, it is a general property of elliptic differential operators, that for the first eigenvalue there is only a one-dimensional array of eigenfunctions, and these are the only ones among all eigenfunctions, which are nonnegative \cite[Section 6.5]{Evans}. Therefore, $\varphi\ge 0$ implies that it can only belong to the very first eigenvalue, and moreover that to this eigenvalue there are no other, non-symmetric eigenfunctions (but only scalar multiples of $\varphi$).

From the properties of the Bessel function we further show that
$$
\frac{\partial \vf}{\partial n}\bigg|_S < 0 ,
$$
where $\dfrac{\partial }{\partial n}$ stands for the normal derivative (in the direction of the outward normal). Indeed, it follows from  $\frac{d}{dz} \left( \frac{J_{\alpha}(z)}{z^\alpha} \right) = - \frac{J_{\alpha + 1}(z)}{z^\alpha}$ \cite[(9.1.30)]{AbramowitzStegun} that
$$
j_\alpha'(z) = - \frac{z}{2(\alpha + 1)} j_{\alpha +1}(z),
$$
which yields
$$
\frac{\partial \vf}{\partial n}\bigg|_S = q_{\frac{d}{2}-1} j_{\frac{d}{2}-1}'(z) \bigg|_{z=q_{\frac{d}{2}-1}} =
- \frac{ q_{\frac{d}{2}-1}^2}{d} j_{\frac{d}{2}}(q_{\frac{d}{2}-1}) < 0,
$$
the latter inequality coming from the fact that $q_{\frac{d}{2}} > q_{\frac{d}{2}-1}$, see \cite[(9.5.2)]{AbramowitzStegun}.

To calculate the Fourier transform of $\vf$, we use the eigenfunction property $-\Delta \vf = q_{\frac{d}{2}-1}^2 \vf$ and obtain
$$
\widehat{\vf}(s):=\int_{\R^d} \vf(x) e^{-isx} dx =-\frac{1}{q_{\frac{d}{2}-1}^2} \int_{\R^d} \Delta \vf(x) e^{-isx} dx,
\quad s \in \R^d.
$$
In view of  $\Delta \exp(-isx) = -|s|^2\exp(-isx)$ we find for the Fourier transform of $\vf$ the formula
$$
(|s|^2 - q_{\frac{d}{2}-1}^2) \widehat{\vf}(s) = \int_B \left( e^{-isx} \Delta \vf (x) - \vf(x) \Delta e^{-isx} \right)  dx.
$$
Next the second Green formula $\int_B \left( f \Delta g - g \Delta f \right) dx = \oint_{S} \left( f \frac{\partial g}{\partial n} - g \frac{\partial f}{\partial n} \right) dx$, here applied with $f(x) := \exp(-isx)$ and $g(x) := \vf(x)$ yields, taking into account the nullity of $\vf|_S$
\begin{align*}
(|s|^2 - q_{\frac{d}{2}-1}^2 ) \widehat{\vf}(s) & = \oint_{S} \left( \exp(-isx) \frac{\partial \vf  (x)}{\partial n} - \vf(x) \frac{\partial \exp(-isx)}{\partial n} \right) dx
\\ & =\oint_{S} \exp(-isx) \frac{\partial \vf (x)}{\partial n} dx
= - \frac{ q_{\frac{d}{2}-1}^2  j_{\frac{d}{2}}(q_{\frac{d}{2}-1})}{d} \oint_{S} e^{-isx} dx.
\end{align*}
The last term $\oint_{S} e^{-isx} dx$ can be interpreted as the Fourier transform of the regular, bounded Borel measure $d\nu(x)$, supported on the unit sphere $S$ and absolutely continuous with respect to surface area measure with the density function $1$, so that
$$
\widehat{\vf}(s) = - \frac{ q_{\frac{d}{2}-1}^2  j_{\frac{d}{2}}(q_{\frac{d}{2}-1})}{d} \ \frac{ \widehat{\nu}(s)}{|s|^2 - q_{\frac{d}{2}-1}^2 }.
$$
It is  well-known (see e.g. \cite[Chapter IV, \S 3]{stein-weiss:Euclidean} or \cite[Appendix B.4]{Grafakos})  that
\begin{equation}\label{Sphere-transform}
\widehat{\nu}(s) := \int_{\R^d} e^{-isx} d\nu(x)
= \frac{2 \pi^{\frac{d}{2}}}{\Gamma \left( \frac{d}{2} \right) } j_{\frac{d}{2}-1}(|s|), \quad s \in \R^d.
\end{equation}

We now introduce the new function
\begin{equation}\label{Y:psidef}
\psi(x) := - \oint_{S} \vf(x-u) \frac{\partial \vf (u)}{\partial n} du
=  \frac{ q_{\frac{d}{2}-1}^2  j_{\frac{d}{2}}(q_{\frac{d}{2}-1})}{d}  \  (\vf * d\nu)(x).
\end{equation}
Note that $\varphi$ is supported in $B$ and $d\nu$ is supported in $S$, whence $\psi$ is supported in $2B$. Its Fourier transform is
\begin{align*}
\widehat{\psi}(s) & = \frac{ q_{\frac{d}{2}-1}^2  j_{\frac{d}{2}}(q_{\frac{d}{2}-1})}{d}  \ \widehat{\vf}(s) \widehat{\nu}(s)
= - \left[ \frac{ q_{\frac{d}{2}-1}^2  j_{\frac{d}{2}}(q_{\frac{d}{2}-1})}{d} \right]^2
\frac{ \left[  \widehat{\nu}(s) \right]^2}  { |s|^2 - q_{\frac{d}{2}-1}^2} \\
& = \left[ \frac{\pi^{\frac{d}{2}} q_{\frac{d}{2}-1} j_{\frac{d}{2}} (q_{\frac{d}{2}-1}) } { \Gamma\left( \frac{d}{2} + 1 \right)} \right]^2 \ \frac{j_{\frac{d}{2}-1}^{2}(|s|)}{1 - |s|^2 / q_{\frac{d}{2}-1}^{2}}.
\end{align*}
Finally, we normalize this function to take the value $1$ at the origin and consider
$$
y_d(s) := \left[ \frac{\pi^{\frac{d}{2}} q_{\frac{d}{2}-1} j_{\frac{d}{2}} (q_{\frac{d}{2}-1}) } { \Gamma\left( \frac{d}{2} + 1 \right) } \right]^{-2}  \ \widehat{\psi}(s)  = Y_d(|s|),   \quad s \in \R^d,
$$
where $Y_d$ is given by \eqref{Yudin-function}.
The property (i) of the function $Y_d$ is obvious. To prove the property (ii), note that $\widehat{y}_d$ coincides with $\psi$ up to a constant. The properties in (ii) easily follow from \eqref{Y:psidef}.
 \eop \end{proof}

\begin{lemma}[Gorbachev]\label{lem-h}
There exist two positive constants $0<\de, \kappa$, and a continuous radial positive definite  function $h \in L^1(\RR^d)$ such that  $\widehat{h}$ is compactly supported, $\widehat{h}(0)=0$, $h(0)=1$, and
\begin{equation}\label{h-cond}
h(x) \le -\frac{\kappa}{|x|^{d+1}},\quad |x|\ge 1-\delta.
\end{equation}
\end{lemma}

\begin{proof} (Gorbachev)
Following \cite{gorbachev:L2} (see properties of the function $G_{\alpha}$), consider the function
$$
H(t) :=   \int_{t}^{\infty} s Y_{d+2}(s) ds,
\quad t \ge 0,
$$
where $Y_{d+2}$ is the Yudin function \eqref{Yudin-function}, and the continuous radial function $h_0$ is defined as $h_0(x) := H(|x|)$, $x \in \R^d$.

First we show that for $t$ large enough the estimate
\begin{equation} \label{H-growths}
H(t) \le -\frac{\kappa_1}{t^{d+1}}
\end{equation}
holds true with a certain constant $\kappa_1 > 0$. From the well-known asymptotic relation \cite[(9.2.1)]{AbramowitzStegun} which for real $t$ takes the form
$$
J_{\alpha}(t) = \sqrt{\frac{2}{\pi t}} \cos{\left( t - \frac{1}{2}\alpha \pi - \frac{1}{4}\pi \right)} + O\left( \frac{1}{t} \right), \quad t \to \infty,
$$
it follows that with $A := \frac{1}{\sqrt{\pi}} 2^{\alpha + \frac{1}{2}} \Gamma(\alpha+1)$ and $\beta := \frac{\alpha \pi}{2} + \frac{\pi}{4}$ we have
\begin{equation} \label{Bessel-j-asymt}
j_{\alpha}(t) = A \frac{1}{t^{\alpha + \frac{1}{2}}} \cos \left( t - \beta \right) + O \left( \frac{1}{t^{\alpha + 1}}\right), \quad t \to \infty.
\end{equation}
Taking into account that $\frac{1}{1-s^2} =  -\frac{1}{s^2} + O \left( \frac{1}{s^4}\right)$, $s \to \infty$, we obtain with a little trigonometry---i.e. applying $\cos^2(\gamma)=\frac12 (1+\cos(2\gamma))$---the formula
$$
s Y_{d+2}(s) =  \frac{s j_{\frac{d}{2}}^2(s)}{1 - {s^2}/{q_{\frac{d}{2}}^2}}
= -\frac{1}{2} A^2 q_{\frac{d}{2}}^2 ~ \frac{1}{s^{d+2}} \bigg(1+\cos( 2s-2\beta)\bigg) +  O \left( \frac{1}{s^{d + \frac{5}{2}}}\right), \quad s \to \infty.
$$
Given a large $t$, integrating the above asymptotic formula we are led to
\begin{align*}
H(t) & =  - \frac{1}{2} A^2 q_{\frac{d}{2}}^2   ~\left( \frac{1}{(d+1)t^{d+1}} + \int_t^\infty \frac{\cos(2 s -2\beta)}{s^{d+2}} ds \right)
+ O \left( \frac{1}{t^{d + \frac32}} \right)
\\ & =  - \frac{A^2  q_{\frac{d}{2}}^2}{(2d+2)} ~ \frac{1}{t^{d+1}}
+ O \left( \frac{1}{t^{d + 2}} \right)
+ O \left( \frac{1}{t^{d + \frac32}} \right)
\le - \frac{\kappa_1}{t^{d +1}}
\end{align*}
with a suitable constant $\kappa_1 > 0$ for $t$ large enough. Note on passing that we found $H(t)=O(t^{-(d+1)})$ and therefore also $h_0\in L^1(\RR^d)$.

Next we will investigate the Fourier transform of $h_0$. It is well-known that the Fourier transform of a radial function $f(x) = F(|x|)$ is again radial, i.e. $\widehat{f}(s) = K(|s|)$, and $F$ and $K$ are connected by the Fourier-Bessel, or Hankel, transform:
\begin{equation} \label{Hankel-connection}
\widehat{f}(s) = K(|s|) = (2\pi)^{\frac{d}{2}} ({\cal H}_{\frac{d}{2}-1} F)(|s|),
\end{equation}
where the Fourier-Bessel transform is defined for $\alpha \ge -\frac{1}{2}$ by the formula
$$
({\cal H}_\alpha F)(s) := \frac{1}{2^\alpha \Gamma(\alpha + 1)} \int_0^\infty F(u) j_\alpha(su) u^{2\alpha + 1} du, \quad s \ge 0.
$$
Note that \eqref{Hankel-connection} follows easily from \eqref{Sphere-transform} by changing to polar coordinates (see e.g., \cite[Chapter IV, \S 3]{stein-weiss:Euclidean}, \cite[Appendix B.5]{Grafakos}).
Using the identity  $\frac{d}{dz} \left( z^\alpha J_{\alpha}(z) \right) = z^\alpha J_{\alpha - 1}(z)$ \cite[(9.1.30)]{AbramowitzStegun}, we arrive at
\begin{equation} \label{j-derivative}
\frac{d}{du} \left( u^{2\alpha} j_{\alpha}(su) \right) = 2\alpha u^{2\alpha - 1} j_{\alpha-1}(su).
\end{equation}

The latter identity and integration by parts yield (taking $\alpha=d/2$ in the above)
\begin{align*}
({\cal H}_{\frac{d}{2}-1} H)(s)
& = \frac{1}{2^{\frac{d}{2}-1} \Gamma \left( \frac{d}{2} \right)}  \int_0^\infty H(u) j_{\frac{d}{2}-1}(su) u^{d-1} du \\
& = \frac{1}{2^{\frac{d}{2}-1} \Gamma \left( \frac{d}{2} \right) d} \int_0^\infty H(u) d \left(  j_{\frac{d}{2}}(su) u^{d} \right) \\
& = \frac{1}{2^{\frac{d}{2}} \Gamma \left( \frac{d}{2} + 1 \right)} \left( \left. H(u) j_{\frac{d}{2}}(su) u^{d} \right|_0^\infty
-  \int_0^\infty H'(u) j_{\frac{d}{2}}(su) u^{d} du \right)\\
& =  \frac{1}{2^{\frac{d}{2}} \Gamma \left( \frac{d}{2} + 1 \right)} \int_0^\infty Y_{d+2}(u) j_{\frac{d}{2}}(su) u^{d+1} du \\
& = ({\cal H}_{\frac{d}{2}} Y_{d+2})(s).
\end{align*}
In the calculation above we used the fact  that the substitution $\left. H(u) j_{\frac{d}{2}}(su) u^{d} \right|_0^\infty$ vanishes (see \eqref{H-growths} for $u \to \infty$), and that $H'(u) = -u Y_{d+2}(u)$. Thus,
\begin{equation} \label{H-Y-connection}
{\cal H}_{\frac{d}{2}-1} H = {\cal H}_{\frac{d}{2}} Y_{d+2}.
\end{equation}
Reflecting back to Lemma~\ref{lem-Y}, applied in dimension $d+2$, we see that the function $y_{d+2}(x) = Y_{d+2}(|x|)$, $x \in \R^{d+2}$, satisfies the properties  $\widehat{y}_{d+2} \ge 0$, $\supp{ \widehat{y}_{d+2} } \Subset 2B^{d+2}$, $\widehat{y}_{d+2}(0) = 0$. So according to \eqref{Hankel-connection} (still applied in dimension $d+2$), we therefore have ${\cal H}_{\frac{d}{2}} Y_{d+2} \ge 0$, $\supp{ {\cal H}_{\frac{d}{2}} Y_{d+2} } \Subset [0,2]$ and $\left( {\cal H}_{\frac{d}{2}} Y_{d+2} \right)(0) = 0$. Whence \eqref{H-Y-connection} furnishes that the same is true for ${\cal H}_{\frac{d}{2}-1} H$, so using \eqref{Hankel-connection} again---but now in dimension $d$---we finally obtain in exactly dimension $d$ the properties
$$
\widehat{h}_0 \ge 0, \quad \supp{\widehat{h}_0} \Subset 2B, \quad  \widehat{ h}_0 (0) = 0.
$$
Note that $h_0(0) > 0$ since $h_0$ is positive definite and not identically zero. Now consider the function
$$
h(x) := \frac{ h_0 \left( q_{\frac{d}{2}} x \right) } {h_0(0)}, \quad x \in \R^d.
$$
Clearly $h$ is an integrable continuous function with $\widehat{h} \ge 0$, $ \supp{ \widehat{h}} \Subset 2 q_{\frac{d}{2}}B$, $\widehat{ h} (0) = 0$  and $h(0)=1$.

It remains to show \eqref{h-cond}.  It follows from the consideration above that
$$
h_0(x) \le -  \frac{\kappa_2}{|x|^{d +1}}
$$
with some constant $\kappa_2 > 0$ for $x$ large enough. It is also easy to see that $H(x)$ is negative increasing for $x \ge q_{\frac{d}{2}} $. So, also $H ( q_{\frac{d}{2}})<0$. Thus by continuity, $H(x)<0$ in a small neighborhood of $q_{\frac{d}{2}}$, too. Therefore,  with a sufficiently small value of $\de>0$, and with a suitable constant $\kappa>0$, property \eqref{h-cond} is satisfied, too.
 \eop \end{proof}

\noindent {\bf Continuation of the Proof of Proposition \ref{prop:Gorbachev}.}
Given $\varepsilon > 0$,  let $g \in \FV (B, \RR^d)$ be a function such that $\int_{\RR^d} g > (1 - \varepsilon) \DD(B)$. Consider its perturbation $f(x) := \frac{1}{1 + \varepsilon} (g(x) + \varepsilon h(x))$, where $h$ is the function from Lemma~\ref{lem-h}. It is easy to see that $f \in \FV (B, \RR^d)$. For its integral we have $\int_{\RR^d} f  > \frac{1 - \varepsilon}{1 + \varepsilon} \DD(B)$. For $|x| \ge 1$ we have, since $g(x) \le 0$,
$$
f(x) \le \frac{\varepsilon}{1 + \varepsilon} h(x) \le  -\frac{\kappa \varepsilon}{(1 + \varepsilon) |x|^{d+1}} \le  -\frac{\kappa \varepsilon}{2 (1 + \varepsilon) |x|^{d+1}},
$$
where $\kappa$ is the constant from  Lemma~\ref{lem-h}.
Since $g(x) \le 0$ for $|x| = 1$ and $g$ is continuous (and thus uniformly continuous in each compact set containing the ball $B$), there exists a small $\delta > 0$ such that $g(x) \le \frac{\kappa \varepsilon}{2} \le \frac{\kappa \varepsilon}{2|x|^{d+1}}$ for $1 - \delta \le |x| \le 1$. Taking $\delta > 0$ so small that also \eqref{h-cond} is fulfilled, we obtain
$$
f(x) \le -\frac{\kappa \varepsilon}{2 (1 + \varepsilon) |x|^{d+1}},\quad |x|\ge 1-\delta.
$$
Replacing $\delta$ by $\min{\{ \varepsilon, \delta \}}$, we finally arrive at the estimate
\begin{equation}\label{f-estim}
f(x) \le -\frac{\kappa \delta}{4 |x|^{d+1}},\quad |x|\ge 1-\delta.
\end{equation}

Let $\eta$ be a radial non-negative positive definite Schwartz function such that $\eta(0)=1$ and $\supp{\eta} \subset  B$.  By the properties of Schwartz functions, there is a constant $C_{\eta} > 0$ such that
$$
\widehat{\eta}(x)\le \frac{C_{\eta}}{|x|^{2d+1}},\quad |x|>0.
$$
For $R > 0$, consider $\varphi_{R}(x) := R^{d}\widehat{\eta}(Rx)$. Clearly,  $\varphi_{R}$ is a radial non-negative positive definite Schwartz function such that $\widehat{\varphi}_{R}(y)=\eta\left( \frac{y}{R} \right)$, $\widehat{\varphi}_{R}(0)=1$, $\supp\widehat{\varphi}_{R}\subset RB$, and
\begin{equation}\label{phi-estim}
\varphi_{R}(x)\le \frac{C_{\eta}}{R^{d+1}|x|^{2d+1}},\quad |x|>0.
\end{equation}

Next take the convolution
$f_{R} := f*\varphi_{R}$. Then $\widehat{f}_{R}=\widehat{f}\,\widehat{\varphi}_{R}$.
It is clear that $f_{R}$ belongs to $C(\mathbb{R}^{d})\cap L^{1}(\mathbb{R}^{d})$
and is a positive definite function with $\supp{\widehat{f}_R} \subset RB$. Further on,
$$
f_{R}(0)=\int_{\mathbb{R}^{d}}f(t)\varphi_{R}(t)\,dt
\le f(0) \int_{\mathbb{R}^{d}}\varphi_{R}(t)\,dt = f(0) = 1
$$
(and, on the other hand, $f_R(0) > 0$ since $f_R$ is positive definite and not identically zero). Thus, the function $F_R := \frac{1}{f_R(0)} f_R$ fulfills all the properties defining the class $\GG_0(B)$ if we show that $\supp{(F_R)_+} \subset B$. We will give a proof of this property a couple of lines  below. With this property at hand, we have $F_R \in  \GG_0(B)$.  For the  integral of the function $F_R$ we have
$$
\int_{\RR^d} F_R = \widehat{F_R}(0) = \frac1{f_R(0)} \widehat{f}(0) \widehat{\varphi}_{R}(0) = \frac1{f_R(0)} \widehat{f}(0) \ge \widehat{f}(0) = \int_{\RR^d} f  > \frac{1 - \varepsilon}{1 + \varepsilon} \DD(B).
$$
Since $\varepsilon > 0$ can be taken arbitrarily small, it follows that $\DD(B) \le  \DD^\GG_0(B)$ .

\bigskip
Thus, to complete the proof of the proposition we need to show that, for a sufficiently large $R$ that depend on $d$, $\delta$, and on the particular choice of $\eta$,  we have  $f_{R}(x) \le 0$, $|x|\ge 1$. Consider $x$ with $|x|\ge 1$. We take $R$ in the form $R = r \delta^{-2}$, where the constant $r > 0$ will be chosen later. In what follows $C$ and $C'$  will denote positive constants that may depend on $d$  and  $\eta$ but do not depend on $f$, $\delta$ and $x$; these constants may be different at different occasions.   We have
\begin{align*}
f_{R}(x)& =\int_{\mathbb{R}^{d}} f(t)\varphi_{R}(x-t) \,dt
\\
& = \left( \int_{|t|<(1-\delta)|x|} + \int_{(1-\delta)|x|\le |t|\le 2|x|} + \int_{|t|>2|x|} \right) f(t)\varphi_{R}(x-t) \,dt
=:  I_{1} + I_{2} + I_{3}.
\end{align*}
If $|t|>2|x|\ge 2$ then $f(t)\le 0$, and thus
$$
I_{3} \le 0.
$$
Since  $f(t) \le f(0) = 1$, we have
$$
I_{1} \le  \int_{|t| < (1-\delta)|x|}\varphi_{R}(x-t)\,dt.
$$
If $|t| < (1-\delta)|x|$, then $|x-t|\ge |x|-|t| > \delta |x|$. This and \eqref{phi-estim} imply
\begin{equation}\label{J1-estim}
I_1
\le  \int_{|t| <  (1-\delta)|x|} \varphi_{R}(x-t)\, dt
\le \frac{C_{\eta}}{R^{d+1}(\delta |x|)^{2d+1}}\int_{|t| < |x|}dt
= \frac{C \delta}{r^{d+1}|x|^{d+1}}.
\end{equation}
If $(1-\delta)|x| \le |t|\le 2|x|$, then also $1-\delta \le |t|$. Hence, by \eqref{f-estim},
$$
I_{2}\le -\frac{C\delta}{|x|^{d+1}}\int_{(1-\delta)|x|\le |t|\le 2|x|}\varphi_{R}(x-t)\,dt.
$$
Further, the ball of radius $\de |x| \ge \de$ around $x$ lies within the domain of integration here, so we get taking into account also \eqref{phi-estim} the estimate
\begin{align*}
\int_{(1-\delta)|x|\le |t|\le 2|x|}\varphi_{R}(x-t)\,dt & \ge \int_{\de B} \varphi_{R}(s)\,ds = 1- \int_{|s|\ge \de} \varphi_{R}(s)\,ds
\\ & \ge 1- \int_{|s|\ge \de} \frac{C_{\eta}}{R^{d+1}|s|^{2d+1}} \,ds
\\ & = 1- \frac{C} {R^{d+1} \de^{d+1}} = 1- \frac{C \de^{d+1}}{r^{d+1}} \ge \frac12
\end{align*}
for large enough $r$. Whence
$$
I_{2} \le -\frac{C'\delta}{|x|^{d+1}}.
$$
Summarizing, we obtain the estimate
$$
f_{R}(x)
\le \frac{C\delta}{r^{d+1}|x|^{d+1}} - \frac{C' \delta}{|x|^{d+1}}
= \frac{\delta}{|x|^{d+1}} \left( \frac{C}{r^{d+1}} -  C' \right),
$$
and finally $f_R(x) \le 0$ if we choose  $r$ so large  that $\frac{C}{r^{d+1}} < C'$.
 \eop \end{proof}

As it is only tangentially touched in the literature (for example, for the particular case of $W=B$ and $Q=r_dB$ using special considerations \cite{gorbachev:sphere}), let us note the following.

\begin{proposition}\label{pr:existence-extremal}
If $W\subset \RR^d$ is closed and it has finite Lebesgue measure $|W|<\infty$ and if $Q\Subset \RR^d$ is compact, then there exists some extremal function $f\in \GG(W,Q)$ with $\int_{\RR^d} f = \DD^\GG(W,Q)$.
\end{proposition}

\begin{proof}
By definition of sup, there are functions $f_n \in \GG(W,Q)$ with $\int_{\RR^d} f_n > \DD^\GG (W,Q)-1/n$.
Further, the family of functions $\GG(W,Q)$ is \emph{equicontinuous}. Indeed, let $R>0$ be such that $Q\Subset RB$. Then for any $f\in \GG(W,Q)$ by Fourier inversion and using $\widehat{f}\ge 0$ and $\widehat{f}(t)=0$ for $t\not\in RB$, we get
\begin{align*}
\left| f(x) - f(y)\right| & =  \frac{1}{(2\pi)^d}  \left| \int_{\RR^d}  \left( e^{ixt}-e^{iyt} \right) \widehat{f}(t) dt \right|
\le  \frac{1}{(2\pi)^d}  \int_{Q} \left|2 \sin\left(\frac{(x-y)t}{2}\right) \right| \widehat{f}(t) dt  \\
& \le \frac{1}{(2\pi)^d}  \max_{t\in RB} |(x-y)t| \cdot \int_{RB} \widehat{f}(t) dt = |x-y|R f(0)= |x-y|R.
\end{align*}
Therefore, for the modulus of continuity of $f\in \GG(W,Q)$ we always have uniformly $\omega(f;h) \le R h ~(h>0)$.

It is then immediate from the celebrated Arzel\`a-Ascoli Theorem that for any $K\Subset \RR^d$ the family of restrictions $g|_K$ of functions $g$ from $\GG(W,Q)$ constitute a precompact set in $C(K)$ (equipped with the maximum norm). Taking say $K_n:=nB$, a standard diagonalization argument furnishes a subsequence of $(f_n)$ converging locally uniformly
to some function $f \in C(\RR^d)$. We can assume that the subsequence itself is $(f_n)$. Then also $\lim f_n = f$ in the pointwise sense, so in view of $f_n\gg 0$ also $f\gg 0$ follows, c.f. \eqref{posdefdef}. Further, $\supp f_{+} \subset \overline{W}=W$ is obvious as $W$ is closed and it also follows that $f(0)=1$ and $|f|\le 1$. Let us write $f=f_+ - f_-$ and similarly $f_n=(f_n)_+-(f_n)_-$. Then we also have $(f_n)_{\pm}\to f_{\pm}$ pointwise. For the negative parts we may apply Fatou's Lemma:
\begin{equation}\label{eq:fnminus}
\int_{\RR^d}  f_- \le \liminf_{n\to\infty} \int_{\RR^d} (f_n)_-.
\end{equation}
For the positive parts note that $(f_n)_+$ and $f_+$ are all supported in $W$, and $|(f_n)_+|\le (f_n)_+(0)=1$, all the functions $f_n$ belonging to $\GG(W,Q)$.
That is, $(f_n)_+ \le \chi_W$, the indicator function of $W$, which is integrable because $|W|<\infty$. Therefore, Lebesgue's Dominated Convergence Theorem yields
\begin{equation}\label{eq:fnplus}
\int_{\RR^d}  f_+ = \lim_{n\to\infty} \int_{\RR^d} (f_n)_+.
\end{equation}
Note that then $\int_{\RR^d} |f|=\int_{\RR^d} f_+ + \int_{\RR^d} f_-  \le \lim_{n\to \infty} \int_{\RR^d} (f_n)_+ ~+~ \liminf_{n\to \infty} \int_{\RR^d} (f_n)_- \le 2 \lim_{n\to \infty} \int_{\RR^d} (f_n)_+ \le 2|W|$ because $\int_{\RR^d} (f_n)_- = \int_{\RR^d} \left( (f_n)_+ - f_n \right) \le \int_{\RR^d} (f_n)_+$ for each $n$, for $\int_{\RR^d} f_n  \ge 0$ in view of $f_n\gg 0$. Therefore, we have also proved $f \in L^1(\RR^d)$, that is, also $f\in C(\RR^d) \cap L^1(\RR^d) \cap L^\infty(\RR^d)$, whence it is also in $L^2(\RR^d)$. In particular, $\widehat{f}$ exists, is continuous, and belongs to $L^2(\RR^d)$.

Now we claim that $f \in \GG(W,Q)$. Almost all requirements of the definition \eqref{eq:Gorbachevclass} were already proved;
to demonstrate $f \in \GG(W,Q)$ it remains only to show  $\supp \widehat{f} \subset Q$.

Before proving this, let us note that all the functions $f_n$ and $f$ belong to the unit ball of $L^\infty(\RR^d)$, which is the dual space of $L^1(\RR^d)$, the space $L^1(\RR^d)$ being separable. By the Banach-Alaoglu Theorem the unit ball in a dual space is weak-star sequentially compact and so for the sequence $(f_n)$ there is a weak-star convergent subsequence---which we may assume to be the very $(f_n)$ here---resulting in $\langle f_n,H \rangle \to \langle f,H \rangle$ for any fixed $H\in L^1(\RR^d)$, the inner product standing for $\langle f,H \rangle:=\int f\overline{H}$, as usual.
(That the limit function in the weak-star sense cannot be else than $f$ itself follows from locally uniform convergence coupled with the availability as a particular choice of $H$ of characteristic functions of any compacts.) We will exploit this weak-star convergence in the following argument.

Let us take any point $z \not\in Q$, and any small enough ball $\delta B$ around zero such that $z + 2 \delta B \subset \RR^d\setminus Q$. Then take $\psi:= \widehat{f} \cdot \theta $ with $\theta(t):=(\chi_{\delta B} * \chi_{\delta B}) (t-z)$. Consider $h(x):=\check{\theta}(x) =e^{izx} (\chi_ {\delta B}* \chi_{\delta B})\check{}(x)  = \frac{1}{(2\pi)^d} e^{izx}  \widehat{\chi}_{\delta B}^2(x) = \frac{\delta^{2d}}{(2\pi)^d} e^{izx} \widehat{\chi}_B^2(\delta x)$.  Here $\ft{\chi}_B$ is the Fourier transform of the characteristic function of $B$, which is well-known  and directly follows from \eqref{Hankel-connection} and \eqref{j-derivative} (see e.g. \cite[Appendix B.5]{Grafakos}):
\begin{equation} \label{Fourier-char-ball}
\ft{\chi}_B(x) = \frac{\pi^{\frac{d}{2}}}{\Gamma\left( \frac{d}{2} + 1 \right)} j_{\frac{d}{2}}(|x|).
\end{equation}
Taking into account \eqref{Bessel-j-asymt}, we see that $h$ is an integrable function on $\RR^d$. Let $g := f  * h$.  By Fourier inversion we obtain $g =\check{\psi}$.
Let us record here that $\supp \theta =z + 2\delta B$ is disjoint from $Q$ by construction.

Fixing a point $y\in \RR^d$ we thus have $g(y) = \int_{\RR^d} f(x) h(y-x) dx$. For the fixed function $H:=\overline{h_y}:=\overline{h(y-\cdot)} \in L^1(\RR^d)$ we have by consideration above $\langle f_n,H \rangle \to \langle f,H \rangle$.  So invoking also the Plancherel formula
we get
\begin{align*}
& g(y) = \lim_{n\to \infty}  \langle f_n,H \rangle =  \frac{1}{(2\pi)^d}  \lim_{n\to\infty} \langle \widehat{f_n}, \widehat{\overline{h_y}} \rangle = \frac{1}{(2\pi)^d}  \lim_{n\to \infty} \int_{\RR^d} \widehat{f_n} \overline{\widehat{\overline{h_y}}} \\
& = \frac{1}{(2\pi)^d}   \lim_{n\to \infty} \int_{\RR^d} \widehat{f_n}(t)  e^{iyt} \theta (t) dt = 0
\end{align*}
for $\widehat{f_n}$ is supported in $Q$, not intersecting with the support of $\theta$. As a result, $g \equiv 0$. In view of the uniqueness of Fourier transform we thus found $\psi \equiv 0$. That is, $\widehat{f} \cdot \theta \equiv 0$, which means that outside $Q$ the continuous  function $\widehat{f}$ must be zero. Therefore, $\supp \widehat{f} \subset Q$ and $f \in \GG(W,Q)$.

Subtracting \eqref{eq:fnminus} from \eqref{eq:fnplus} and using the definition of the extremal constant, we get
$$
\DD^\GG(W,Q) \ge \int_{\RR^d} f \ge \limsup_{n\to \infty} \int_{\RR^d} \left((f_n)_+-(f_n)_-\right) = \limsup_{n\to \infty} \int_{\RR^d} f_n \ge \DD^\GG(W,Q),
$$
whence we have equality everywhere here and $f$ is thus an extremal function.
 \eop \end{proof}

\begin{remark}\label{rem:LCA-existence}
While our paper waited printing, a nice and nontrivial generalization of Proposition \ref{pr:existence-extremal} to general locally compact Abelian groups appeared in \cite{Zsuzsi-Marcell}.
\end{remark}

Lastly, let us turn to the works of Cohn-Elkies and Viazovska. They use the function class
\begin{eqnarray*}
\EK (W_+, W_-):=\left\{ f \in \FF_1(W_+, W_-)  \right.  :  |f(x)|=O((1+|x|)^{-\kappa})   \textrm{~and~} \\
|\widehat{f}(x)| = O((1+|x|)^{-\kappa}) \left. \right\}
\end{eqnarray*}
with some (arbitrarily given) $\kappa:=d+\de>d$ and $W_+, W_- \subset \RR^d$ chosen to be $W_+:=B$ and $W_-:=\RR^d$. Their extremal constant is thus $\DE(B)$, where in general $\DE (W_+) := \CEK (W_+,\RR^d)$ and
$$
\CEK (W_+, W_-):= \sup \left\{ \int_{\RR^d} f : f\in \EK(W_+, W_-) \right\} \qquad (W_+, W_- \subset \RR^d).
$$
Already there is a dependence on a parameter $\kappa$ or $\delta$ here---but we do not discuss directly that for different parameters the constants $\CEK (W_+, W_-)$ coincide. Instead, we prove

\begin{proposition}
 For any parameter value $\kappa >d$ and for any open sets $W_+, W_- \subset \RR^d$, we have for the extremal constants that $\CCC(W_+, W_-)=\CEK(W_+, W_-)$. In particular, $\DD(W_+)=\DE(W_+)$, where $\DE(W_+)=\DEK(W_+)$ is in fact independent from the choice of the parameter~$\kappa$.
\end{proposition}

\begin{proof}
Again we only need to consider the case when $0 \in W_+$. Observe that
$$
\CEK(W_+, W_-) \le \CCC_1(W_+, W_-)=\CCCC(W_+, W_-),
$$
for $\EK(W_+, W_-) \subset \FF_1(W_+, W_-)$ by definition. It remains to see that $\CEK(W_+, W_-)\ge \CCCC(W_+, W_-)$.

So let $\ve>0$ be arbitrary, choose $n> \kappa/(d+1)$ a natural number, and let $f\in \FC(W_+, W_-)$ be any function with $\int_{\RR^d} f > \CCCC(W_+, W_-) - \ve$. As $\supp f_{+} \Subset W_+$ and $W_+$ is open, we have for an appropriately small neighborhood of 0---say with an appropriately small closed ball $\eta \overline{B}$ of radius $\eta>0$---that $\supp f_{+} + 2n\eta \overline{B} \Subset W_+$ still holds.

Take now $u$ to be the $2n$-th convolution power of the characteristic function of $B$, i.e. $u:=\underbrace{\chi_B * \dots *\chi_B}_{{2 n ~\textrm{times}}}$. As $\chi_B=\widetilde{\chi}_B$, this is the Boas-Kac square of the $n$-th convolution power of $\chi_B$, and as such, it is a positive-definite function; moreover, $\ft{u}=|\ft{\chi}_B|^{2n}$. Taking into account  \eqref{Fourier-char-ball} and \eqref{Bessel-j-asymt}, we obtain $|\ft{u}(x)|=O(|x|^{-n(d+1)})~(x\to \infty)$.

It remains to scale $u$ to our needs: we want a function $v:=c u(\lambda x)$ such that it be supported in $2n\eta \overline{B}$ and satisfy $\int_{\RR^d} v =1$. Obviously, $\lambda:=1/\eta$ and $c:=\eta^{-d} |B|^{-2n}$ will do: then $\supp v \Subset 2n\eta \overline{B}$ and its integral is normalized to 1, whence for $g:=f * v$ we find $\int_{\RR^d} g = \int_{\RR^d} f  \int_{\RR^d} v = \int_{\RR^d} f > \CCCC(W_+, W_-) - \ve$. Also, $\supp g_{+} \Subset \supp f_{+} + \supp v = \supp f_{+} + 2n\eta \overline{B} \Subset W_+$ and $\supp g_{-} \Subset \RR^d$. Therefore, $g$ satisfies all conditions to belong to $\FC(W_+, W_-)$ but for the normalization $g(0)=1$. We have $g(0) =  \int_{\RR^d} f(x) v(-x) dx \le  \int_{\supp{f_+}} f(x) v(-x) dx \le f(0)  \int_{\supp{f_+}} v(-x) dx \le  \int_{\RR^d} v(x) dx = 1$. Therefore, taking $h:=\frac{1}{g(0)} g$ we finally get $h \in \FC(W_+, W_-)$ and $\int_{\RR^d} h(x) dx > \CCCC(W_+, W_-)-\ve$. Noting that by scaling and multiplying by a constant the defining property of the decrease of the Fourier transform was not spoiled, so we also have $|\ft{v}|= O(|x|^{-n(d+1)})~(x\to \infty)$.
Finally, $\ft{h}= \frac{1}{g(0)} \ft{f} \ft{v}$ shows that the same ordo estimate remains in effect also for $\ft{h}$ (as $|\ft{f}|$ is bounded by say $|\supp f| f(0)= |\supp f|$). In all, we find that $h \in \EK(W_+, W_-)$, and so $\int_{\RR^d} h \le \CEK(W_+, W_-)$. It follows that $\CEK(W_+, W_-) \ge \CCCC(W_+, W_-)-\ve$. As $\ve>0$ could be fixed arbitrarily, we finally obtain $\CEK(W_+, W_-)\ge \CCCC(W_+, W_-)$.
 \eop \end{proof}


\section{Homomorphisms and the extremal problem} \label{sec:homom}

In this section we obtain statements about the behavior of the value ${\cal C}_G(\Omega_+,\Omega_-)$ under homomorphisms. We follow the considerations of \cite{KR-2006}.

Let $G$ and $H$ be two LCA groups, and let $\varphi : G \to H$ be a continuous group homomorphism onto $H$. The kernel of this homomorphism $K := \Ker{(\varphi)}=\varphi^{-1}(0)$ is a closed subgroup of $G$, and thus it is a LCA group itself. We consider the quotient group $G/K$ together with the canonical or natural projection $\pi : G \to G/K$ which maps an element $g \in G$ to its coset, i.e., $\pi(g) := [g] := g + K \in G/K$. By the definition of the topology on $G/K$, $\pi$ is an open and continuous mapping. Moreover, $\psi := \varphi \circ \pi^{-1} : G/K \to H$ is a continuous isomorphism of the LCA groups $G/K$ and $H$. For details, see, e.g., \cite[Appendices B.2 and B.6]{Rudin-book}.

The Haar measure of a group is determined up to a constant factor. However, the choice of this factor influences the value   ${\cal C}_G(\Omega_+,\Omega_-)$. Suppose the Haar measures $m_G$ and $m_H$ are given. As is standard, we will choose the Haar measures on $K$ and $G/K$ such that $dm_G = dm_K \, dm_{G/K}$, c.f. \cite[(2) on page 54]{Rudin-book}. The isomorphism $\psi$ leads in a natural way to another Haar measure $\nu_H$ on $H$ defined by $\nu_H := m_{G/K} \circ \psi^{-1} = m_{G/K} \circ \pi \circ \varphi^{-1}$. But two Haar measures are constant multiples of each other, so we can define the constant $M := \frac{dm_H}{d\nu_H}$.

\begin{theorem}
\label{homomorphism-theorem}
Let $G$ and $H$ be LCA groups considered with the Haar measures $m_G$ and $m_H$, and let  $\varphi : G \to H$ be a continuous open group homomorphism onto $H$. Let the Haar measures of the subgroup $K := \Ker{(\varphi)}$ and of the quotient group $G/K$ be normalized such that $dm_G = dm_K \, dm_{G/K}$. Let $\nu_H := m_{G/K} \circ \pi \circ \varphi^{-1}$,  where $\pi : G \to G/K$ is the natural projection, and let  $M := \frac{dm_H}{d\nu_H}$.

Let $\Omega_+$ and $\Omega_-$ be open, $0$-symmetric subsets of $G$,
and let $\Theta_\pm := \varphi(\Omega_\pm) \subset H$. Then
$$
{\cal C}_G(\Omega_+,\Omega_-) \le
\frac{1}{M} \, {\cal C}_H(\Theta_+,\Theta_-) \, {\cal C}_{K}(\Omega_+ \cap K, \Omega_- \cap K).
$$
\end{theorem}

This result corresponds to \cite[Proposition 3]{KR-2006} and also the proof goes along the same lines. However, we need to point out that for this proof one really needs to assume that $\vf$ is an \emph{open} continuous homomorphism, somewhat restricting generality of both statements here and in \cite{KR-2006}.

\begin{proof}
The sets $\Theta_\pm$ and $\Omega_\pm \cap K$ are obviously open in the corresponding topologies of $H$ and $K$, respectively, and $0$-symmetric. Clearly, $0 \not\in \Omega_+$ if and only if $0 \not\in \Omega_+ \cap K$, and in this case both sides of the inequality are zero. We therefore consider the case $0 \in \Omega_+$. Then also $0 \in \Omega_+ \cap K$ and $0 \in \Theta_+$.
The mapping $\psi := \varphi \circ \pi^{-1} : G/K \to H$ is a continuous open isomorphism of the LCA groups $G/K$ and $H$.

For each $h \in H$ choose (invoking here the Axiom of Choice) $g(h) \in G$ to be an arbitrary representative of the inverse image $\varphi^{-1}(h)$, i.e. an element of the coset $\psi^{-1}(h)$. Let $f \in {\cal F}_{c,G}(\Omega_+,\Omega_-)$ where the notation emphasizes that we consider a function $f$ on $G$. Define $F : H \to \R$ by
$$
F(h) := \int_K f(g(h) + k) \, dm_K(k).
$$
Now we claim that $F$ is continuous, too. To show this, take an arbitrary $\varepsilon > 0$: then by uniform continuity of $f$ there exists a neighborhood $V=V_G$ of $0$ in $G$ such that $|f(g_1) - f(g_2)| < \varepsilon$ for all $g_1, g_2 \in G$ with $g_1 -  g_2 \in V_G$. Clearly $V_G$ can be taken an open set with compact closure so that in particular $m_G(V_G) < \infty$. Let us fix $h_1 \in H$ and  write $g_1:=g(h_1) \in \varphi^{-1}(h_1)$: we are to show that $F$ is continuous at $h_1$. Put $V_H := \varphi(V_G)$. Since $\varphi$ is open and so is $V_G$, also $V_H$ is a neighborhood of $0$ in $H$. Consider $h_2 \in H$ such that $h_1 - h_2 \in V_H$. This means that $h_1 - h_2 = \varphi(g)$ for some $g \in V_G$. Let $g_2:=g(h_2) (\in \varphi^{-1}(h_2))$. Then $\varphi(g_1 - g_2) = \varphi(g_1) - \varphi(g_2) =h_1- h_2 = \varphi(g)$.  It follows that $(g_1 - g_2) - g \in \Ker(\varphi)=K$, i.e.  $g_1 - g_2 - g = k^*$ with some $k^* \in K$. For the element $g_2^* := g_2 + k^*$ we have $\varphi(g_2^*) =h_2$ and $g_1 - g_2^* = g \in V_G$. By the choice of the neighborhood $V_G$ we have $|f(g_1 + k) - f(g_2^* + k)| < \varepsilon$ for all $k \in K$. Thus,
\begin{align*}
\Big|F(h_1) - F(h_2)\Big| & = \left| \int_K f(g_1+k)\, dm_K(k) - \int_K f(g_2+k) \, dm_K(k) \right|
\\& = \left| \int_K \Big( f(g_1+k)- f(g_2+k^*+k) \Big) \, dm_K(k) \right|
\\&  \le  \int_K \Big|f(g_1 + k) - f(g_2^* + k)\Big| \, dm_K(k) <  C \,\varepsilon
\end{align*}
with $C:=  m_K( (\supp{f} - g_1 + V_G) \cap K) < \infty$. (Here we had to use that $f$ is compactly supported, thanks to $f\in {\cal F}_{c,G}(\Omega_+,\Omega_-) $.) This shows that $F$ is continuous at $h_1$.

Now we show that $\supp{F_\pm} \Subset \Theta_\pm$. Indeed, $\{ h~:~ F(h)>0\} \subset \{ h ~:~ \exists k \in K ~\textrm{such that}~ f(g(h)+k)>0\} = \vf(\{g~:~f(g)>0\}) \subset \vf(\supp f_{+})$, which is compact, whence also for the closure $\supp F_{+} \Subset \vf(\supp f_{+}) \subset  \vf(\Omp) =\Theta_{+}$.
The proof for $\supp{F_-}$ is similar.

Clearly, $F(0) = \int_K f(k) \, dm_K(k)$. By Fubini's theorem,
\begin{eqnarray} \label{F-integral}
\int_H F(h) \, dm_H(h) & = &
\int_H \int_K f(g(h) + k) \, dm_K(k) \, M \, d\nu_H(h) \nonumber\\
& = &
M \, \int_H \int_K f(g(h) + k) \, dm_K(k) \, dm_{G/K}(\psi^{-1}(h)) \nonumber \\
& = &
M \, \int_G f(g) \, dm_G(g).
\end{eqnarray}

To prove that $F$ is positive definite on $H$, we first notice that for every continuous character $\chi$ on $H$ the function $\gamma := \chi \circ \varphi$ is a continuous character on $G$. What we are to use here is that for a continuous and integrable function positive definitness is equivalent to non-negativity of its Fourier transform. This follows from the inversion theorem for the Fourier transform \cite[Theorem 1.5.1]{Rudin-book} and the Bochner-Weil theorem \cite[Theorem 1.4.3]{Rudin-book}. See also \cite[(4.23) Corollary]{Folland} or \cite[p.~483]{KR-2006}.
Applying (\ref{F-integral}) to $f_1 := f\gamma$  and
$$
F_1(h) := \int_K f_1(g(h) + k) dm_G(k) = \int_K f(g(h) + k) \chi(\vf(g(h) + k)) dm_G(k) = \chi(h) F(h),
$$
we obtain
$$
\widehat{F}(\overline{\chi}) =  \int_H F(h) \, \chi(h) \, dm_H(h) = M \, \int_G f(g) \, \gamma(g) \, dm_G(g) = M \widehat{f}(\overline{\gamma}) \ge 0
$$
since $f $ is positive definite on $G$. Thus, $\widehat{F}\ge 0$ and in view of the continuity and integrability of $F$, this implies $F\gg 0$ on $H$.

We see from the above that the function $F_0 := \frac{1}{F(0)} F = \frac{1}{\int_K f \, dm_K} F$ belongs to the class ${\cal F}_{c,H}(\Theta_+,\Theta_-)$, and thus
$$
\int_H F \, dm_H \le {\cal C}_H(\Theta_+,\Theta_-) \, \int_K f \, dm_K.
$$
Furthermore, $f|_K$ is positive definite and thus $f|_K \in {\cal F}_{c,K}(\Omega_+ \cap K, \Omega_- \cap K)$, giving the estimate $\int_K f \, dm_K \le {\cal C}_K(\Omega_+ \cap K, \Omega_- \cap K)$. Using (\ref{F-integral}), we obtain
$$
\int_G f \, dm_G = \frac{1}{M} \, \int_H F \, dm_H \le  \frac{1}{M} \, {\cal C}_H(\Theta_+,\Theta_-) \, {\cal C}_K(\Omega_+ \cap K, \Omega_- \cap K)
$$
for each $f \in {\cal F}_{c,G}(\Omega_+,\Omega_-)$, which implies the desired statement.
 \eop \end{proof}

\begin{corollary}
Let $\varphi : G \to G$ be a continuous open automorphism of a LCA group $G$, and let $\Omega_+, \Omega_-$ be two open\footnote{In view of the existence of an open continuous automorphism, the two copies of $G$ have equivalent topologies.}, $0$-symmetric sets in $G$.
 Then
$$
{\cal C}_G(\varphi(\Omega_+),\varphi(\Omega_-)) = M \, {\cal C}_G(\Omega_+,\Omega_-) ,
$$
where  $M := \dfrac{d (m_G \circ \varphi)}{d m_G}$. In particular, if   $0<m_G(\Omega_{\pm})<\infty$ then we have $M=\frac{m_G(\varphi(\Omega_+))}{m_G(\Omega_+)} = \frac{m_G(\varphi(\Omega_-))}{m_G(\Omega_-)}$.

\end{corollary}

\begin{proof} (Cf. \cite[Corollary 1]{KR-2006}.)
We apply Theorem~\ref{homomorphism-theorem} with $H=G$. In this case $K = \{0\}$, $m_K = \delta_0$, $K \cap \Omega_{+} \subset \{0\}$, and in case $K \cap \Omega_{+} = \{0\}$ we have ${\cal C}_K(\Omega_+ \cap K, \Omega_- \cap K) = 1$, $G/K = G$, $m_{G/K} = m_G$, and $\pi$ is the identity.  (The other case when $K \cap \Omega_{+} =\emptyset$ is trivial and gives that both sides vanish.) To calculate the constant $M$ we observe that $\nu_G = m_G \circ \varphi^{-1}$, and for each measurable set $\Omega^* \subset G$ we have $M = \frac{m_G(\Omega^*)}{m_G(\varphi^{-1}(\Omega^*))}$. The desired representation of $M$ can be obtained by taking $\Omega^* = \varphi(\Omega_+)$ and  $\Omega^* = \varphi(\Omega_-)$, respectively.

Now, Theorem~\ref{homomorphism-theorem} gives
$$
{\cal C}_G(\Omega_+,\Omega_-) \le \frac{1}{M} \, {\cal C}_G(\varphi(\Omega_+),\varphi(\Omega_-)).
$$
Since $\varphi^{-1}$ is also a continuous  open automorphism, an application of Theorem~\ref{homomorphism-theorem} for $\varphi^{-1}$ provides the converse inequality.
 \eop \end{proof}

\begin{corollary}
Let $G_1, \ldots, G_n$ be LCA groups and $G := G_1 \times \cdots \times G_n$. Let $\Omega_{j,\pm} \subset G_j$, $j=1, \ldots, n$, be open, $0$-symmetric sets,
and let $\Omega_\pm := \Omega_{1,\pm} \times \cdots \times \Omega_{n,\pm}$. Then
\be \label{tensor-product}
{\cal C}_G(\Omega_+,\Omega_-) \le {\cal C}_{G_1}(\Omega_{1,+},\Omega_{1,-}) \cdots  {\cal C}_{G_n}(\Omega_{n,+},\Omega_{n,-}).
\ee

\end{corollary}

\begin{proof}
(Cf. \cite[Corollary 2]{KR-2006}.) The inequality follows by induction in $n$ from Theorem~\ref{homomorphism-theorem} with $\varphi$ being a projection to one of the components of the direct product\footnote{In \cite[Corollary 2]{KR-2006}, the corresponding statement for the Tur\'an problem was considered, and it was shown that in this case, i.e., when $\Omega_+ = \Omega_-$, (\ref{tensor-product}) turns into equality. However, we cannot guarantee equality in the current general case. This arises from the fact that the product of two negative functions can be positive.}.
 \eop \end{proof}

The reader will find no difficulty in extending the above to the topological product of an arbitrary number of LCA groups. However, openness of the sets $\Omega_\pm$ imply that apart from a finitely many initial components the rest of the groups $G_i$ are contained in $\Omega_\pm$, whence the corresponding extremal constants ${\cal C}_{G_i}(\Omega_{i,+},\Omega_{i,-})$ are either $1$ in case when $G_i$ is compact, or infinity in case when $G_i$ is not compact.

\begin{corollary}
Let $G$ be a LCA group, $K$ be a closed subgroup of $G$, and suppose that the Haar measures are normalized so that $dm_G = dm_K \, dm_{G/K}$. Let $\pi : G \to G/K$ denote the natural projection. If $\Omega_+, \Omega_-$ are two open, $0$-symmetric sets in $G$,
then
$$
{\cal C}_G(\Omega_+,\Omega_-) \le  {\cal C}_{G/K}(\pi(\Omega_+),\pi(\Omega_-)) \, {\cal C}_{K}(\Omega_+ \cap K, \Omega_- \cap K).
$$
\end{corollary}

\begin{proof}
(Cf. \cite[Corollary 3]{KR-2006}.) Apply  Theorem~\ref{homomorphism-theorem} with $H = G/K$ and $\varphi = \pi$ which is open and continuous. In this case $\nu_H = m_{G/K} \circ \pi \circ \varphi^{-1} = m_{G/K}$, so that $M=1$.
 \eop \end{proof}


\section{Packing, covering, tiling and the extremal problem}

Let $H \in {\cal B}_0$.  We say that the set $H$ \emph{packs} $G$ \emph{by translation} with the translation set $\Lambda \subset G$ if
\be \label{packing}
\sum_{\lambda \in \Lambda} \chi_H(x - \lambda) \le 1 \qquad \text{a.e.} \ x \in G.
\ee
In other words, for a.e. $x \in G$ there is at most one $\lambda \in \Lambda$ such that $x$ lies in the set $H + \lambda$.

Further on, we say that the set $H$ \emph{covers} $G$ \emph{by translation} with the translation set $\Lambda \subset G$ if
$$
\sum_{\lambda \in \Lambda} \chi_H(x - \lambda) \ge 1 \qquad \text{a.e.} \ x \in G.
$$
In other words, $H + \Lambda$ contains almost all points of $G$.

Finally, we say  that the set $H$ \emph{tiles} $G$ \emph{by translation} with the translation set $\Lambda \subset G$ if $H$ simultaneously packs and covers $G$ with the translation set $\Lambda$, i.e.,
\begin{equation}\label{tiling}
\sum_{\lambda \in \Lambda} \chi_H(x - \lambda) = 1 \qquad \text{a.e.} \ x \in G.
\end{equation}
This means that almost all $x \in G$ belong to exactly one of the sets $H + \lambda$.

The slight generalization using a.e. conditions here (rather than the strict conditions for every point $x\in G$) became widely used for its convenience when dealing with tiling: for a closed square we want to say that it still packs (so also tiles) space, and the same way we also consider the open squares still covering (so even tiling) space.

In general it is tacitly assumed, see e.g. in \cite{KR-2006} or \cite{Revesz-2011}, that we can always ``correct'' the underlying set $H$ by a measure zero difference to become a strict packing or covering or tiling, as we wish. Therefore, if we need to apply e.g. a strict packing condition, then we may modify the setup accordingly.

This is indeed true essentially, whence in the further discussion we will feel free to require the following somewhat more stringent conditions, which we indeed need in the proofs. Namely we will consider the assumption that inequality (\ref{packing}) is fulfilled \emph{for all} $x \in G$. When this holds, we will say that $H$ \emph{packs} $G$ \emph{in the strict sense} (and in case it also covers $G$ with the same $\Lambda$, we will accordingly say that $H$ \emph{tiles} $G$ \emph{in the strict sense}). So we say that $H$ \emph{tiles} $G$ \emph{in the strict sense}, if $H$ tiles $G$ and the packing is in the strict sense, i.e., if the tiling is disjoint---but we still do not assume the covering to hold everywhere (but only a.e.).

It is easy to see that this packing condition in the strict sense is equivalent to  \eqref{stricpacking}. This motivates---closely following \cite{KR-2006} and \cite{Revesz-2011}---the consideration of the following ``generalized packing type condition'', where already there is no packing, but a general set $W$ replaces the difference set $H - H$ of the packing set $H$ in the above formulation.

\begin{definition}\label{def:genpack}
We say that a set $W \in {\cal B}_0$ satisfies a generalized strict packing type condition (``packing \emph{type} condition'' for short) with the translation set  $\Lambda \subset G$ if
$$
(\Lambda - \Lambda) \cap W \subseteq \{0\}.
$$
\end{definition}

Note that difference sets have many strong structural properties, which are extensively analyzed in the literature, see e.g. \cite{matolcsi-ruzsa} and the references therein, so replacing a difference set by a general set $W$ without this extra structure is indeed a generalization.

Also, reflecting back to the original setup, it is worth noting that even if packing by $H$ or $M\subset H$ of the same measure can be equivalent, the difference sets $H-H$ and $M-M$ may indeed have essentially different properties. Before proceeding let us see an instructive example, explaining why we step back from the a.e. formulation.

\begin{example}\label{ex:difference} Consider $G:=\RR$ and let $H:=\{-4\} \cup (-1,1)\cup\{4\}$, which satisfies a (not strict) packing (and also covering and tiling) condition with the translational set $\Lambda:=2\ZZ$. If we ``correct'' $H$ by dropping the two isolated points to become $M:=(-1,1)$, then $M$ already satisfies a strict packing (and tiling) condition with the same $\Lambda=2\ZZ$. However, the difference sets $Q:=H-H=(-5,-3)\cup(-2,2)\cup(3,5)$ and $W:=M-M=(-2,2)$ are totally different. Indeed, $W$ satisfies the (generalized, strict) packing type condition of Definition \ref{def:genpack} with $\La=2\ZZ$, while the same fails for $Q$: in fact $(\La-\La) \cap Q = 2\ZZ \cap Q = \{-4, 0, 4\}$. Further, if any translational set $L$ satisfies $(L-L)\cap Q =\{0\}$, then the asymptotic density of $L$ cannot exceed $2/5$ (while the asymptotic density of $\Lambda$ was $1/2$). Furthermore, even the Delsarte constants of the two difference sets are essentially different:
\begin{equation} \label{eq:example1}
\DD(Q)>\DD(W)=2.
\end{equation}

\end{example}


{\bf Proof of~\eqref{eq:example1}.}
Let $L\subset \RR$ be an arbitrary set satisfying $(L-L)\cap Q =\{0\}$. Assume (as we may translate $L$) that $0\in L$. Let us list the positive elements of $L$ in increasing order: $L_{+}:=L\cap(0,\infty) =\{\ell_1<\ell_2<\dots<\ell_n<\dots\}$. As $0\in L$, we must have $\ell_1\ge 2$, for $\ell_1-0 \in L-L$ cannot belong to $Q$. The same holds for any consecutive pairs $\ell_{k+1}$ and $\ell_k$: we must have $\ell_{k+1}-\ell_k \ge 2$. But adding this for two consecutive differences we find $\ell_{k+2}-\ell_k \ge 4$, and, as $(3,5)\subset Q$, we infer even $\ell_{k+2}-\ell_k\ge 5$. It follows that $\ell_2\ge 5$ and in general $\ell_{2k}\ge 5k$. Arguing similarly for $L_{-}:=L\cap(-\infty,0)$, we find that the number of points of $L$ lying in $[-5n,5n]$ can be at most $4n+1$, furnishing the upper estimate $2/5$ for the asymptotic density of $L$. (This can indeed be attained by choosing $L:=5\ZZ\cup (2+5\ZZ)$.)

The value of the Delsarte constant is easier to find for $W=M-M$, as $M$ tiles in the strict sense with $\Lambda=2\ZZ$, whence the below Proposition \ref{tile-theorem} provides $\DD(W)=2$.

In the following we estimate $\DD(Q)$ showing that it exceeds $2$: our construction in fact will even prove that $\CCC(Q,\emptyset)>2$.

To start with, let $T(t):=1+a\cos t + b \cos{(4t)} \ge 0$ be any nonnegative cosine polynomial with spectrum $\{0,1,4\}$: we denote the set of all such polynomials by $\PP$, say. Also recall that the usual triangle function $\Delta(t):=(1-|t|)_{+}=(\chi_{[-1/2,1/2]} * \chi_{[-1/2,1/2]})(t)$ is positive definite, $\Delta\gg 0$.

So now consider the measure $\mu:=\de_0+\dfrac{a}{2}(\de_1+\de_{-1})+\dfrac{b}{2}(\de_{4}+\de_{-4})$ (with $\de_c$ standing for the Dirac measure concentrated on the point $c$). The Fourier transform of this measure is exactly $T(t)$, whence if $T\ge 0$ (i.e. when $T\in \PP$), then $\mu\gg0$ is a positive definite measure, whence the convolution $\Phi:=\Delta * \mu$ is also a positive definite continuous function (with Fourier transform $T(t)\cdot\left(\frac{\sin t}{t}\right)^2 \ge 0$). Note that $\Phi(0)=\int \Delta(-x) d\mu(x) = \Delta(0)=1$. Further, $\Phi \ge 0$ and the support of $\Phi$ is contained in the sum of the supports of $\De$ and $\mu$ from its defining convolution, i.e. in $[-1,1]+\{0,\pm 1, \pm 4\} =[-5,-3]\cup[-2,2]\cup[3,5]=\overline{Q}$. Therefore (essentially) we obtain $\Phi \in \FC(Q,\emptyset)$ entailing $\CCC(Q,\emptyset)\ge \int \Phi(x)dx = \int \int \Delta(x-y) d\mu (y) dx = (\int \Delta) \cdot \mu(\RR)= 1+a+b$. (Here we neglected a trivial dilation---taking $\Phi_{\ve}:= \Delta_{\ve} * \mu$ with $\De_{\ve}(x):=\De((1+\ve)x)$ and then passing to the limit when $\ve\to+0$ could precisely show the same.)

It already proves the assertion if we find a cosine polynomial $T\in \PP$ with $T(0)=1+a+b>2$. Existence of such a polynomial, on the other hand, is kind of trivial, for the minimum of $\cos t$ is at $\pi$ (modulo $2\pi\ZZ$), while there the wave $\cos(4t)$ is strictly positive: so a polynomial $1+\cos t +\ve \cos(4t)$ must be nonnegative for small enough (but still positive) $\ve>0$.

A more precise analysis of such ``trinomials'' has been carried out (also for being applied in a more intricate question through duality) in \cite{R-Szeged}. To find an explicit (possibly close to be best) value we briefly employ the methods of \cite[\S 2]{R-Szeged} here.

\begin{lemma}\label{l:trinom} We have $\sup \left\{ 1+a+b~:~ T(t)=1+a\cos t + b \cos(4t) \in \PP\right\} >2$. In fact, the value $T(0)=2.236...$ is achieved by the (approximately extremal) polynomial $T_0(t):= 1+ 0.989286995... \cos t + 0.246780732...  \cos(4t)$.
\end{lemma}
\begin{proof} First note that $F(T):=T(0)=1+a+b$ is a linear functional on $C(\TT)$, whence according to \cite[Lemma 2.3]{R-Szeged} its maximum on $\PP$ is attained on some cosine polynomial from the set ${\mathcal Z}:=\{1-\cos(4t)\} \cup {\mathcal Z}_0 \cup {\mathcal Z}_\pi$ with
\begin{align*}
{\mathcal Z}_0:= & \{h_z(t):=1+a(z)\cos t +b(z) \cos(4t) ~:~ 0\le z\le\pi/4\}, \qquad{\rm where} \\  &a(z):=\frac{4\sin(4z)}{d(z)},\quad b(z):=\frac{\sin z}{d(z)},\quad d(z):=4 \cos z \sin(4z)-\cos(4z)\sin z,
\end{align*}
(interpreting the coefficients of $h_0$ by their limits as $z\to 0$), and ${\mathcal Z}_\pi:= \{h_z(x+\pi)=1-a(z)\cos t +b(z) \cos(4t) ~:~ h_z \in {\mathcal Z}_0\}$.

Here it is immediate that $\max F$ on $\PP$ is positive, whence is not attained on $1-\cos(4t)$; also, as the expression for $a(z)$ is nonnegative for all $0\le z\le \pi/4$, it is clear that $h_z(0)\ge h_z(\pi)$ and the maximum of the functional $F$ on $\PP$ is attained on ${\mathcal Z}_0$. On this set
$$
F(h_z)=h_z(0)=1+a(z)+b(z)= 1+\dfrac{4 \sin{(4z)} +\sin z}{4\cos z \sin{(4z)} - \cos{(4z)} \sin z}, $$
so in particular $F(h_0)=h_0(0)=1+\frac{16}{15}+\frac{1}{15}=2\frac{2}{15}>2$ already. We could only find the extremum numerically: the optimal value is $z=0.628...$ where we get the above.
 \eop \end{proof}

\medskip
Using the above Lemma \ref{l:trinom} we finally find $\DD(Q) \ge \CCC(Q,\emptyset) = 2.23606...>2$, as claimed.
 \eop

In view of the above the reader may have doubts what the ``measure zero correction'' can achieve. We now formally state and prove what one can certainly do in this regard. Although the proof is standard, it is quite tedious and more complicated than one would expect it after such a clear heuristical meaning.

\begin{proposition}\label{prop:smallcorrection} Assume that the set $H\in\B$ packs $G$ with the translation set $\Lambda$. Then there exists $M\subset H$ with $m_G(H\setminus M)=0$, such that $M$ packs $G$ with $\La$ in the strict sense.
\end{proposition}
\begin{proof} There is nothing to prove if $m_G(H)=0$, as then $M=\emptyset$ suffices. So let us consider the case when $m_G(H)>0$.  Also assume, as we may, $0\in H$.

Let the exceptional set in \eqref{packing} be $X:=\{ x\in G~:~ \sum_{\lambda \in \Lambda} \chi_H(x - \lambda) > 1\}$, in other words $X=\{ x\in G~:~ \exists \lambda \ne \lambda' \in \Lambda  ~~\textrm{such that}~~x \in (H+\lambda)\cap (H+\lambda')\}$. Let $\ve>0$ be arbitrary. By assumption, $m_G(X)=0$, meaning that for the given $\ve>0$ there exists an open set $U\supset X$ with $m_G(U)<\ve$. So let us fix also the open set $U$.

Recall that $H\in \B$, whence $\overline{H}\Subset G$. Consider the compact set $K:=\overline{H}-\overline{H}$ and the compactly generated subgroup $G_0:=\langle K \rangle$. As $H$ is of positive Haar measure, its difference set---whence $K$, too---contains a neighborhood of $0$, see e.g. \cite[Corollary 20.17]{HewittRossI}. Therefore, $G_0$ is also a neighborhood of 0 and $G_0$ is thus an open subgroup, whence a compactly generated open-closed subgroup, too.

The following fact  follows from \cite[2.4.2. Lemma]{Rudin-book} and its proof, as well as from parts of the proof of  \cite[2.6.7. Theorem]{Rudin-book}, see also the proof of Theorem 7 in \cite{Revesz-2011}.

\begin{lemma} \label{structure-lemma}
Let $K$ be a compact neighborhood of $0$ in $G$, and let $G_0 := \langle K \rangle$ be the subgroup in $G$ generated by $K$. Then \\
1) There is a finitely generated lattice $L$ in $G_0$ (isomorphic to $\Z^d$ with some $d \in \N$) such that $K \cap L = \{0\}$.\\
2) There is a set $E \in {\cal B}_0$ such that $K \subset E \subset G_0$ and each $x \in G_0$ can be uniquely represented as $x = e + \ell$, where $e \in E$ and $\ell \in L$. In particular, $E$ tiles $G_0$ in the strict sense with the translation set $L$.\\
If $G$ is not compact, then necessarily $d \ge 1$.

\end{lemma}

Now let $E$ and $L$ be as in Lemma~\ref{structure-lemma}.  So each translated copy $\ell+E$ with $\ell \in L$ has compact closure, whence even in $\ell + \overline{E}$ there can only be some finitely many points of $\Lambda$ (if we check that $\Lambda$ must be discrete---this will be done in Lemma \ref{Lambda-discrete-lemma} below), and then of course $\Lambda_0:=\Lambda\cap G_0$ is countable. (Note that this may well fail for the whole of $G$.)

The subgroup $G_0$ partitions $G$ into conjugate classes: let us take such a (disjoint) partition $G=\cup_{t\in T}(G_0+t)$ where $T$ is a representative set (chosen using the Axiom of Choice) of inequivalent conjugate classes (translates) of $G_0$ within $G$. Let us then take the decomposition $U= \cup_{t \in T} ((t+G_0)\cap U)$.

Any of these $U(t):=(t+G_0) \cap U $ is the intersection of open sets, whence open, and either has a positive measure, or is empty. As for different $t$ these are also disjoint, there can only be a set $T^* \subset T$ of at most countably many $t$ with $(t +G_0) \cap U \ne \emptyset$---compare \cite[(2.22) Proposition]{Folland}. So, we find $X \subset \cup_{t \in T^*} ((t+G_0) \cap U) = \cup_{i=0}^\infty U_{i}$, with $U_{i}:=((t_i+G_0) \cap U)$ open and $T^*=\{t_i~:~i\in \NN\}$ countable. (We can assume $t_0 \in G_0$ in numbering these representatives of conjugate classes.)

Now let us consider the countable set $\Lambda_0:=\Lambda \cap G_0$ together with its peers $\Lambda_{i}:=\Lambda \cap (t_i+G_0)$ for arbitrary $i \in \NN$. By the same reason as for $\Lambda_0$, all $\Lambda_{i}$ are (at most) countable, whence so is the set $\Lambda^*:= \bigcup_{i=0}^\infty \Lambda_{i}$. Note that for $\lambda \not \in \Lambda^*$ and with $t \in T$ such that $\lambda \in t + G_0$ we have $(\lambda + G_0) \cap X = (t + G_0) \cap X \subset (t + G_0) \cap U = \emptyset$ by construction.

Taking the set $Y:=\cup_{i \in \NN} \cup_{\lambda\ne\lambda' \in \Lambda_{i}} (\lambda-\lambda' +H)$, the whole union is an at most countable union, while each member $H \cap (\lambda-\lambda' +H)$ is obviously of measure zero (since $(H+\lambda)\cap (H+\lambda') \subset X$). So finally defining $C:=H \cap Y = \cup_{i \in \NN} \cup_{\lambda\ne\lambda' \in \Lambda_{i}} \left( H \cap (\lambda-\lambda' +H)\right)$ as ``correction set'', this is a countable union of measure zero sets and is thus of measure zero. Therefore, $M:=H\setminus C$ has $m_G(M)=m_G(H)$.

Consider now any $z \in (\Lambda-\Lambda) \cap (H-H)$, and assume that $z\ne 0$: we want to show that then $z \not\in (\Lambda-\Lambda) \cap (M-M)$. By condition, $z=\lambda-\lambda'=h-h'$ with $\lambda \ne \lambda' \in \Lambda$, $h \ne h' \in H$. Take $x:=z+\lambda'+h'=\lambda+h' = \lambda'+h$. Then with the different $\lambda \ne \lambda'$ we find $\chi_H (x-\lambda)=\chi_H (x-\lambda')=1$ and $\sum_{\lambda \in \Lambda} \chi_H (x-\lambda) >1$. This means that $x\in X \subset U = \cup_{i=1}^\infty U_i$. So let $x \in U_{i}$ for some $i\in \NN$. Then $x \in t_i+G_0$, so that $\lambda'=x-h$ and also $\lambda=x-h'$ belong to the same coset $t_i+G_0$, i.e. $\lambda, \lambda ' \in \Lambda_{i}$.  (Here we have used the inessential assumption that $0\in H$; otherwise there would occur some other coset $t_{j}+G_0$ and the corresponding $\Lambda_j$ here.) Therefore, $\lambda-\lambda'+h'=h$ belongs to $(\lambda-\lambda'+H)\cap H \subset C$ and so $h \not \in M =H\setminus C$. It follows that in the representation $z=h-h'$ for $z$ we cannot have $h \in M$, whence $z \not\in (\Lambda-\Lambda) \cap (M-M)$, as wanted. This proves $(\Lambda-\Lambda) \cap (M-M) \subset \{0\}$.
 \eop \end{proof}

\medskip

We have already used the basic observation that a translational set $\Lambda$ (of a packing by translates of some $H\in \B$ of positive measure) is necessarily discrete, i.e. to any point $g\in G$ there is a neighborhood $U$ of $g$ with $\Lambda \cap U \subset \{g\}$. (Equivalently, we may say that $\Lambda\cap U$ is finite, or, equivalently again, we may formulate discreteness with postulating that $\Lambda \cap K$ is finite for every compact set $K\Subset G$.) So we prove this now.

\begin{lemma} \label{Lambda-discrete-lemma}
Let the open set $0\in W \subset G$ satisfy a strict packing type condition with the set $\Lambda$, i.e. assume $W \cap(\La-\La)=\{0\}$. Then $\Lambda$ is discrete.
\end{lemma}

\begin{proof}
Consider an arbitrary $K\Subset G$ and a set of $\ell$ different points $\lambda_1,\dots,\lambda_\ell$ in $\Lambda \cap K$. Consider an open neighborhood $V$ of zero, with compact closure, such that $V-V \subset W$. (Such $V$ exists, for $0\in\intt W$ and the group operation is continuous.) The strict packing type condition on $W$ and $\La$ implies that $V$ satisfies the strict packing condition \eqref{stricpacking}, i.e. $(V-V) \cap (\La-\La) =\{0\}$.

Therefore, we also have for any compact subset $C \Subset V$ that
\begin{align*}
m_G(C + K) & \ge \int_{C + K} \sum_{\lambda \in \Lambda} \chi_V(x - \lambda) dx \ge \int_{C + K} \sum_{j=1}^\ell \chi_C(x -\lambda_j) dx \notag \\ & = \sum_{j=1}^\ell \int_{C + K} \chi_C(x - \lambda_j) dx \ge \sum_{j=1}^\ell \int_{C + \lambda_j} \chi_C(x - \lambda_j) dx = \ell m_G(C).
\end{align*}
Note that $C$ and $K$ being compact, so is $C + K$, whence its Haar measure is finite. This shows that for an arbitrary $C \Subset V$ with $m_G(C)>0$ we must have $\ell \le\frac{m_G(C + K)}{m_G(C)} <\infty$, and $\ell$ is bounded.
 \eop \end{proof}

Our next aim is to calculate ${\cal C}(\Omega_+,\Omega_-)$ in the case when $\Omega_+$ is a difference set of a strict lattice tile with a finitely generated lattice.


\begin{proposition} \label{tile-theorem}
Let $G$ be a LCA group. Suppose that $\Omega_+, \Omega_-$ are open, $0$-symmetric sets and
$$
\Omega_+ = H - H,
$$
where $H \in {\cal B}_0$ and $H$ tiles $G$ in the strict sense of {\rm (\ref{stricpacking})}  with the translation set $\Lambda \subset G$ which is a finitely generated lattice.
Then
$$
{\cal C}(\Omega_+, \Omega_-) = m_G(H).
$$
\end{proposition}

The result above is quite special. In principle, we are interested in connecting the Delsarte constant and the packing density in the more general situation when the positivity set $\Omega_+$ is not supposed to have the structure of a difference set like in Proposition~\ref{tile-theorem}, and when the translational set is not assumed to have the structure of a finitely generated lattice. Instead, we only request that $\Omega_+$ satisfies the strict packing type condition and $\Lambda$ has only an asymptotic uniform upper density. Finally the much stronger result---Theorem~\ref{main-theorem}---will be proved about this general situation, but Proposition ~\ref{tile-theorem} will be an indispensable auxiliary result in our argumentation.

The statement of Proposition~\ref{tile-theorem} is known for the Tur\'an problem \eqref{Turan-problem}. In the special case when $G = \R^d$ and  $\Omega$ is itself a convex lattice tile (so that $H = \Omega/2$ can be taken in the proposition), the statement was proved in  \cite{arestov:tiles} (see also \cite{hexagon}). The same follows from \cite{kolountzakis:turan} where the result was obtained for convex $\Omega$ that are spectral (which is the case for all convex tiles). The analogous  proposition in the general form as below for Tur\'an's problem in compact Abelian groups as well as in the groups $\R^d$ and $\Z^d$ was obtained in \cite{KR-2006}. Finally, it was proved for locally compact Abelian groups in  \cite{Revesz-2011}.

%

\bigskip

{\bf Proof of Proposition~\ref{tile-theorem}.}
Since the constant ${\cal C}(\Omega_+, \Omega_-)$ is monotone in the second argument $\Omega_-$, we have
$$
{\cal C}(\Omega_+, \emptyset) \le {\cal C}(\Omega_+, \Omega_-) \le {\cal C}(\Omega_+, G).
$$
The estimate ${\cal C}(\Omega_+, \emptyset)  \ge m_G(H)$ has been proved in \cite[Corollary 5]{Revesz-2011}. For completeness, first we briefly repeat here the construction that shows this inequality, as it was done in the proof of \cite[Corollary 5]{Revesz-2011}: then our proof will be completed by proving $m_G(H) \ge {\cal D}(\Omega_+) = {\cal C}(\Omega_+, G)$ below.

The proof of the inequality ${\cal C}(\Omega_+, \emptyset)  \ge m_G(H)$ uses an old idea how to construct a function giving a lower bound in the investigation of the Tur\'an constant. Assume as we may that $m_G(H) > 0$. Let $A \Subset H$ with $m_G(A) > 0$. Consider the function $f := \chi_A * \widetilde{\chi_A}$. Then $f \gg 0$, $f \ge 0$ on $G$ and $\supp{f} \subset A - A \Subset H - H \subset \Omega_+$. Clearly, $f(0) = m_G(A)$ and $\int_G f  = m_G(A)^2$. Thus, $f_0 := \frac{1}{f(0)}f \in \FC(\Omega_+,\emptyset)$, whence ${\cal C}(\Omega_+, \emptyset) \ge \int_G f_0 = m_G(A)$. Since $H$ is a Borel set, its measure can be approximated arbitrarily closely by measures of inscribed compact sets $A$. Taking the supremum over all such sets $A$, we obtain the desired estimate.

Lastly, we prove $m_G(H) \ge {\cal D}(\Omega_+)$.
We take an arbitrary $f \in \FC(\Omega_+,G)$. Denote $W := \supp{f} \Subset G$.
Let us consider the function
$$
F(x) := \sum_{\lambda \in \Lambda} f(x + \lambda), \qquad x \in G.
$$
The sum is well-defined, for each $x \in G$ only finitely many summands are non-zero, because if $f(x+\lambda)\ne 0$ then $x+\lambda \in W$, too, i.e. $\lambda \in W-x$, which is a compact set while $\Lambda$ is discrete.

Further, the set $L \subset \Lambda$ of points $\lambda\in \Lambda$ with the additional property that $x+\lambda \in W$ for some $x \in H$ must be a finite set. Indeed, if $x+\lambda \in W$, then $\lambda \in W-H \subset W-\overline{H}$; but in view of the assumption $H \in \B$, the latter is compact and so by discreteness $\Lambda \cap (W-H)$ must be finite. So in particular $f(x+\lambda)=0$ for all $x\in H$ and for any $\lambda \in \Lambda \setminus L$. Thus we have that $F(x) = \sum_{\lambda \in L} f(x + \lambda)$ for all $x \in H$.

Next we show that $F$ is positive definite. We need to show that for all $M \in \N$, $x_1, \ldots, x_M \in G$ and $c_1, \ldots, c_M \in \C$ we have
$$
\sum_{k =1}^M \sum_{k'=1}^M c_k \, \overline{c_{k'}} \, F(x_k - x_{k'})
= \sum_{k =1}^M \sum_{k'=1}^M c_k \, \overline{c_{k'}} \, \sum_{\lambda \in \Lambda} f(x_k - x_{k'} + \lambda) \ge 0.
$$
Let the lattice $\Lambda$ be generated by the elements $a_1, \ldots, a_d \in G$, i.e., $\Lambda =  \{\lambda_\nu := \nu_1 a_1 + \cdots + \nu_d a_d : \nu = (\nu_1,\ldots,\nu_d) \in \Z^d \}$. Since there are finitely many points $x_k - x_{k'}$ in the sum above and $f$ is compactly supported, there is a finite set $I \subset \Z^d$ such that
\be \label{F5}
\sum_{k =1}^M \sum_{k'=1}^M c_k \, \overline{c_{k'}} \, F(x_k - x_{k'})
= \sum_{k =1}^M \sum_{k'=1}^M c_k \, \overline{c_{k'}} \, \sum_{\mu \in I} f(x_k - x_{k'} + \lambda_\mu).
\ee
We have to show that this expression is non-negative.

Take $N \in \N$ such that $I \subset \{-N,\ldots,N\}^d$. Using the property of the positive definiteness of $f$ for the $(N+1)^d \, M$ points $x_k + \lambda_\nu$, $\nu \in \{0, \ldots, N\}^d$, $k = 1, \ldots, M$, and taking the coefficient corresponding to a point $x_k + \lambda_\nu$ to be $c_k$, we obtain
\begin{eqnarray*}
0 &  \le & \sum_{k=1}^M \sum_{k'=1}^M  \sum_{\nu  \in \{0, \ldots, N\}^d}  \sum_{\nu'  \in \{0, \ldots, N\}^d}
c_k \, \overline{c_{k'}} \, f(x_k + \lambda_\nu - x_{k'} - \lambda_{\nu'})
\\
& = &  \sum_{k=1}^M \sum_{k'=1}^M c_k \, \overline{c_{k'}} \sum_{\mu  \in \{-N, \ldots, N\}^d} \prod_{j=1}^d (N + 1 - |\mu_j|) \, f(x_k - x_{k'} + \lambda_\mu)
\\
& = &  \sum_{k=1}^M \sum_{k'=1}^M c_k \, \overline{c_{k'}} \sum_{\mu  \in I} \prod_{j=1}^d (N + 1 - |\mu_j|) \, f(x_k - x_{k'} + \lambda_\mu).
\end{eqnarray*}
It follows that
$$
\sum_{k=1}^M \sum_{k'=1}^M c_k \, \overline{c_{k'}} \sum_{\mu  \in I} \frac{\prod_{j=1}^d (N + 1 - |\mu_j|)}{(N + 1)^d}  \, f(x_k - x_{k'} + \lambda_\mu)
\ge 0.
$$
Taking limit when $N \to \infty$, we obtain the desired statement about the non-negativity of the expression (\ref{F5}).

As $F$ is positive definite, it holds $F(x) \le F(0)$ for all $x \in G$.  Consequently,
\be \label{F1}
\int_H F \le F(0) \, m_G(H).
\ee
For the integral we have
\begin{align}\label{F2}
\int_H F & =  \int_H \sum_{\lambda \in L} f(x + \lambda) ~dm_G(x) = \sum_{\lambda \in L}   \int_H f(x + \lambda) ~dm_G(x) \notag \\  &= \sum_{\lambda \in L}  \int_G \chi_{H+\lambda}(y) f(y) ~dm_G(y) = \int_W \left( \sum_{\lambda \in L}  \chi_{H+\lambda}(y) \right) f(y) ~dm_G(y) \\ \notag & = \int_W \left( \sum_{\lambda \in \Lambda}  \chi_{H+\lambda}(y) \right) f(y) ~dm_G(y) = \int_{W} 1 \cdot f(y) ~dm_G(y) = \int_G f,
\end{align}
using that $H$ tiles with $\Lambda$ so that \eqref{tiling} applies.

As we have the strict packing condition \eqref{stricpacking} we must have
$f(\lambda) \le 0$ for all $\lambda\ne 0$ and $\lambda \in \Lambda$; hence,
\be \label{F3}
F(0) = \sum_{\lambda \in \Lambda} f(\lambda) \le f(0) = 1.
\ee
From (\ref{F1}), (\ref{F2}) and (\ref{F3}) we obtain
$$
\int_G f = \int_H F \le F(0) \, m_G(H) \le m_G(H),
$$
which proves the inequality ${\cal D}(\Omega_+) \le m_G(H)$.
 \eop 


\section{Asymptotic uniform upper density on LCA groups}\label{sec:auud}

Asymptotic uniform upper density was introduced in $\RR$ and $\ZZ$ by Kahane \cite{Kahane1, KahaneThese, KahaneAnnIF, KahaneLivre}, in two equivalent forms. Later, only the second form became widespread. See also \cite{beurling, Landau} and \cite{Groemer}. In particular, it seems that this density was first used in connection with packings by Groemer, as a special case of the characteristics he considered, see  \cite{Groemer}.
The notion was later called after various names like Beurling, Nyquist, even Banach \cite{Landau, GKS, Fusi, Grekos}.

The notion was, however, not extended to LCA groups for long. A concept has been introduced by the second named author \cite{UAUDens}, \cite{dissertation}, 
see also \cite{Revesz-2011}. An equivalent construction occurred also in the paper \cite{GKS}. We already know several equivalent definitions and constructions and the notion seems to be handy for applications \cite{dissertation}.

\begin{definition}\label{def-auud}
Let $G$ be a LCA group, and $m_G$ the Haar measure. Let $\nu$ be another measure on $G$ with the $\sigma$-algebra of measurable sets ${\cal S}$. The asymptotic uniform upper density (a.u.u.d.) of the measure $\nu$ is then defined by
$$
\overline{D}(\nu,m_G) := \inf_{C \Subset G} { \sup_{V \in {\cal S} \cap {\cal B}_0} { \frac{\nu(V)}{m_G(C+V)}  }}.
$$
In particular, if $\Lambda \subset G$ is a discrete set and $\gamma_{\Lambda} := \sum_{\lambda \in \Lambda} \delta_{\lambda}$ is the counting measure of $\Lambda$, then
\begin{equation} \label{dens-discr-def}
\overline{D}^{\#}(\Lambda) := \overline{D}(\gamma_\Lambda,m_G)
=  \inf_{C \Subset G} { \sup_{V \in {\cal B}_0} { \frac{ \#(\Lambda \cap V) }{m_G(C+V)}  }}.
\end{equation}
\end{definition}

For the motivation for this definition and properties of the a.u.u.d., see \cite{UAUDens}. In particular, if $G = \R^d$ and the Haar measure on $\R^d$ is normalized to be equal to the volume $|\cdot|$, then---as is stated in \cite{Revesz-2011} as Proposition 1 
and fully proved in \cite{UAUDens} as Theorem 1 and in \cite{dissertation} as Theorem 3.1---we have the following. For every convex body $K \subset \R^d$ with unit volume $|K|=1$ the a.u.u.d. $\overline{D}(\nu,m_G)$ of the measure $\nu$ on $\R^d$ coincides with the classical notion of asymptotic uniform upper density of the measure $\nu$ with respect to $K$ defined as
$$
\overline{D}_K(\nu) :=  \limsup_{r \to \infty} { \frac{\sup_{x \in \R^d}{\nu(rK + x)}}{|rK|} }.
$$
The latter is a natural generalization of the frequently used notion of the asymptotic uniform upper density of a measurable set $A \subset \R^d$ with respect to $K$ defined as
$$
\overline{D}_K(A) :=  \limsup_{r \to \infty} { \frac{\sup_{x \in \R^d}{|A \cap (rK + x)| }}{|rK|} },
$$
on the one hand,  and of the asymptotic uniform upper density of a discrete set
$\Lambda \subset \R^d$ with respect to $K$ defined originally by Kahane as
$$
\overline{D}^{\#}_K(\Lambda) :=  \limsup_{r \to \infty} { \frac{\sup_{x \in \R^d}{\#(\Lambda \cap (rK + x)) }}{|rK|} },
$$
on the other hand.

Recall that the so-called (upper, asymptotic) \emph{center density} $\overline{\delta}_{a}(\Lambda)$ of a ball packing $\{B+\lambda \}_{\lambda \in\Lambda}$ in $\RR^d$ is defined as
$$
\overline{\delta}_a(\Lambda):=\limsup_{r\to \infty} {\frac{\#(\Lambda \cap [-r,r]^d)}{2^dr^d}},
$$
while the (upper, asymptotic) \emph{ball packing density} is defined analogously by considering
$$
\overline{\Delta}_a(\Lambda):=\limsup_{r\to \infty} {\frac{|(B+\Lambda)\cap [-r,r]^d|}{2^dr^d}},
$$
the proportionality of space covered. 
It is easy to see that $\overline{\Delta}_{a}(\Lambda) = \omega_d \overline{\delta}_{a}(\Lambda) =\frac{\pi^{\frac{d}{2}}}{\Gamma\left( \frac{d}{2} + 1 \right)}\overline{\delta}_{a}(\Lambda)$.

Following ideas of Groemer \cite{Groemer}, asymptotic densities in sphere packing were replaced by uniform asymptotic densities, the modification meaning taking first a supremum with respect to translations\footnote{It is easy to see that the limits always exist because the quantities $\sup_x\ldots$ are essentially decreasing as $r\to \infty$.}:
$$
\overline{\delta}(\Lambda):=\lim_{r\to \infty} \sup_{x\in \RR^d} {\frac{\#(\Lambda \cap ([-r,r]^d+x))}{2^dr^d}}, \quad
\overline{\Delta}(\Lambda):=\lim_{r\to \infty} \sup_{x\in \RR^d} {\frac{|(B+\Lambda)\cap ([-r,r]^d+x)|}{2^dr^d}}.
$$

In our above notation $\overline{D}_K(\nu)$, the first one corresponds to putting unit masses at the center points in $\Lambda$, and considering the atomic measure $\sum_{\lambda \in \Lambda} \delta_\lambda$ as $\nu$, while the second one means taking the absolutely continuous measure with density function $\sum_{\lambda\in\Lambda} \chi_{B+\lambda}$; in both, the convex body is the unit cube $[-1/2,1/2]^d$ (with unit volume).

Given that the definition of asymptotic \emph{uniform} upper density involves taking also a supremum, it is clear that $\overline{\delta}_{a}(\Lambda) \le \overline{D}^{\#}_K(\Lambda)$, $K$ being an arbitrary convex body with $|K| = 1$. It is well-known, see e.g. \cite{cohn:packings}, that the maximal density of ball packing is the same with respect to both the asymptotic and the uniform asymptotic densities, because arbitrarily close approximation to the extremal density can be constructed by a periodic packing; for the same reason, upper density can be replaced by simply density in any respective statements.

Some connections between the a.u.u.d. and structural properties such as packing, covering and tiling have been established in \cite{Revesz-2011}. We quote three results from this paper.

\begin{proposition}\label{densitybymH}
{\rm (\cite[Proposition 2]{Revesz-2011})}
Assume that $H \in {\cal B}_0$, $\Lambda \subset G$ and  $H$ packs $G$ in the strict sense with the translation set $\Lambda$, i.e., $(H - H) \cap (\Lambda - \Lambda) = \{0\}$.  Then
$$
\overline{D}^{\#}(\Lambda) \le \frac{1}{m_G(H)}.
$$

\end{proposition}

\begin{proposition}
{\rm (\cite[Proposition 3]{Revesz-2011})}
Assume that $H \in {\cal B}_0$, $\Lambda \subset G$ and $H$ covers $G$ with the translation set $\Lambda$. Then
$$
\overline{D}^{\#}(\Lambda) \ge \frac{1}{m_G(H)}.
$$

\end{proposition}

\begin{corollary} \label{density-corollary}
{\rm (\cite[Corollary 3]{Revesz-2011})}
Assume that $H \in {\cal B}_0$, $\Lambda \subset G$ and $H$ tiles $G$ in the strict sense with the translation set $\Lambda$. Then
$$
\overline{D}^{\#}(\Lambda) = \frac{1}{m_G(H)}.
$$

\end{corollary}

We will also need the following simple result.

\begin{proposition} \label{finite-dens-proposition}
If $G$ is compact with the Haar measure normalized such that $m_G(G) = 1$, and if $\Lambda$ is finite, then $\overline{D}^{\#}(\Lambda) = {\#}\Lambda$.

\end{proposition}

\begin{proof}
Taking $C = G$ in the definition of a.u.u.d (\ref{dens-discr-def}), we obtain
$$
\overline{D}^{\#}(\Lambda) \le  \sup_{V \in {\cal B}_0} { \frac{ \#(\Lambda \cap V) }{m_G(G)}  }
= \sup_{V \in {\cal B}_0} { \#(\Lambda \cap V) } = \# \Lambda.
$$
On the other hand, taking $V = G$ in (\ref{dens-discr-def}), we get
$$
\overline{D}^{\#}(\Lambda) \ge  \inf_{C \Subset G} { \frac{ \#(\Lambda \cap G) }{m_G(C+G)}}
= \frac{ \#(\Lambda \cap G) }{m_G(G)} = \# \Lambda.
$$
 \eop \end{proof}

Finally, let us note a simple consequence of the above Proposition \ref{densitybymH}.

\begin{proposition}\label{densityfinite} If the open set $0\in\Omega \subset G$ satisfies a strict packing type condition with the translational set $\Lambda$, then $\La$ is discrete, moreover, the a.u.u.d. of its counting measure is finite: $\overline{D}^{\#}(\Lambda) <\infty$.
\end{proposition}

\begin{proof} Discreteness was given in Lemma \ref{Lambda-discrete-lemma}. As it is done there, we pick an arbitrary open neighborhood $W$ of $0$ satisfying $(W-W) \cap (\La-\La)=\{0\}$. Then Proposition \ref{densitybymH} applies and we find $\overline{D}^{\#}(\Lambda) \le 1/m_G(W) <\infty$ (for $m_G(W)>0$ in view of openness).
 \eop \end{proof}


\section{The strict packing type condition and the Delsarte extremal problem}\label{sec:Delsarte}


%

In this section we will prove Theorem~\ref{main-theorem}.

The statement corresponding to Theorem~\ref{main-theorem} for the Tur\'an extremal problem has been obtained in \cite{KR-2006} in the cases when $G$ is a compact Abelian group or $G$ is one of the groups $\R^d$ and $\Z^d$, and in  \cite{Revesz-2011} in the general case of a LCA group $G$.

\bigskip

\noindent {\bf Proof of Theorem~\ref{main-theorem}.}
1) First we consider the case when $G$ is a compact Abelian group with the Haar measure normalized such that $m_G(G) = 1$. Due to Proposition~\ref{finite-dens-proposition}, it is enough to show that $\int_G f \le \frac{1}{\#\Lambda}$ for each function $f \in \FC(\Omega_+,G)$. This can be done by almost verbatim repetition of the proof of Theorem 2 in \cite{KR-2006}. We give the details for the sake of completeness.

Since $G$ is compact and $\Lambda$ is discrete by Lemma~\ref{Lambda-discrete-lemma}, $\Lambda$ is finite. Take an arbitrary $f \in \FC(\Omega_+,G)$. Consider the function
$$
\Phi(x) := \sum_{\lambda \in \Lambda} \sum_{\lambda' \in \Lambda} f(x + \lambda - \lambda'), \qquad x \in G.
$$
For this derived function $\Phi = f * \delta_{\Lambda} * \widetilde{\delta}_{\Lambda}$, it is easy to see that $\Phi \gg 0$, see e.g. \cite[(32.8) (d)]{HewittRossII}. Further,
\begin{equation}\label{eq:integralPhi}
\int_G \Phi = (\#\Lambda)^2 \int_G f,
\end{equation}
and, since  $(\Lambda - \Lambda) \cap \Omega_+ = \{0\}$ and $f(0)=1$,
$$
\Phi(0) = \#\Lambda \cdot f(0) + \sum_{\lambda \in \Lambda} \sum_{{\lambda' \in \Lambda}\atop{\lambda' \ne \lambda}} f(\lambda - \lambda') \le \#\Lambda.
$$
It follows from the positive definiteness of $\Phi$ that $\Phi(x) \le \Phi(0)$, $x \in G$, and thus
$$
\int_G \Phi \le \Phi(0) m_G(G) \le \#\Lambda \cdot m_G(G) = \#\Lambda,
$$
which yields the desired estimate, when compared to \eqref{eq:integralPhi}.

2) Now we consider the less trivial case when $G$ is not compact. Our proof uses ideas from  the proof of Theorem 7 in \cite{Revesz-2011}. Assume, as we may, $\overline{D}^{\#}(\Lambda) >0$ (as otherwise there is nothing to prove).

Fix $\al > 0$ satisfying $\al< {\cal D}(\Omega_+)$. There is a function $f \in \FC(\Omega_+,G)$ such that
$$
\int_G f > \al.
$$
Take $K$ to be a compact neighborhood of $0$ such that $\supp{f_+} \subset \intt K \subset K = \overline{K}$. Let $G_0 =  \langle K \rangle$, and let the lattice $L$ (isomorphic to $\Z^d$, $d \ge 1$) and the set $E \in {\cal B}_0$ be as described in Lemma~\ref{structure-lemma}. For $n = (n_1, \ldots, n_d) \in \Z^d$, define $\|n\| := \max_{j=1,\ldots,d}{|n_j|}$. Let $V_0 := E$ and for $N \in \N$
$$
L_N := \{ \ell \in L : \|\ell\| \le N \}
\qquad \text{and} \qquad
V_N := E + L_N \subset G_0.
$$
Clearly, $V_N$ tiles $G_0$ in the strict sense with the lattice $(2N + 1)L$.

We now establish that there is an $s \in \N$ such that $E + E \subset V_s = E + L_s$.

Indeed, if $(E + E) \cap (E + \ell) \ne \emptyset$, then  for an element $x \in (E + E) \cap (E + \ell)$ we have $x = e_1 + e_2 = e_3 + \ell$ with $e_1, e_2, e_3 \in E$.  Thus, $\ell = e_1 + e_2 - e_3 \in E + E - E$. Since $E + E - E$ has a compact closure and $L$ is discrete, the set $(E + E - E) \cap L$ is finite and thus  $(E + E - E) \cap L \subset L_s$ with some $s \in \N$. Then also $E + E \subset G = E + L=\cup_{\ell \in L} (E+\ell)$ and thus $E + E \subset \cup_{\ell \in L} ((E + E)\cap (E+\ell))$ implies that for any $\ell$ occurring with a nonempty set here on the right hand side we must have $\ell \in L_s$, whence also $E + E \subset E + L_s$.

We will need the following statement from  \cite{Revesz-2011}.

\begin{lemma} \label{density-lemma}
{\rm (\cite[Lemma 2]{Revesz-2011})}
Let $G$ be a LCA group, and let $\nu$ be a Borel measure on $G$ with $\overline{D}(\nu,m_G) =: \rho > 0$. For each $V \in {\cal B}_0$ and for each $\varepsilon > 0$, there exists $z \in G$ such that
$$
\nu(V + z) \ge (\rho - \varepsilon) \, m_G(V).
$$

\end{lemma}

We apply this lemma to the counting measure of the translation set $\Lambda$ with $\rho :=  \overline{D}^{\#}(\Lambda)  > 0$, and with an arbitrary $0<\ve<\rho$.

Consider the set $V_N$ with a large $N \in \N$. The set $V_N$ has compact closure, hence, $\#(\Lambda \cap (V_N + z)) < \infty$ for each $z \in G$. Take $z \in G$ like in Lemma~\ref{density-lemma}, i.e.,
\begin{equation}\label{M-below}
M := \#(\Lambda \cap (V_N + z))  \ge (\rho - \varepsilon) \, m_{G}(V_N) =  (\rho - \varepsilon) \, m_{G_0}(V_N)
\end{equation}
with fixing the Haar measure of $G_0$ as the restriction of $m_G$ to $G_0$.
Let $\Lambda' := \Lambda \cap (V_N + z) = \{ \lambda_1, \ldots, \lambda_M \}$. Define
$$
\Phi(x) := \sum_{\lambda \in \Lambda'} \sum_{\lambda' \in \Lambda'} f(x + \lambda - \lambda'), \qquad x \in G.
$$
Since the sum consists of finitely many summands, $\Phi$ is well-defined and continuous, moreover,
$\Phi \gg 0$ as above. Further on, $\supp{\Phi} \Subset G$ since it is closed and $\supp{\Phi} \subset \supp{f} + \Lambda' - \Lambda'$ which is a compact set. Thus also $\supp{\Phi_+} \Subset G$. For the latter we have $\supp{\Phi_+}  \subset \supp{f_+} + \Lambda' - \Lambda' \subset  \supp{f_+} + (V_N + z) - (V_N + z) \subset E + V_N - V_N \subset V_{N + s} - V_{N + s}$.  In particular, $\supp{\Phi_+}  \subset G_0$. Thus, $\Phi_0 := \frac{1}{\Phi(0)} \, \Phi|_{G_0} \in {\cal F}_{c,G_0}(V_{N + s} - V_{N + s}, G_0)$.  Since $V_{N + s}$ tiles $G_0$ in the strict sense with the translation set $(2N + 2s +1) \, L$ which is a finitely generated lattice, we can apply Proposition~\ref{tile-theorem} and obtain
\begin{equation}\label{eq:intPhizero}
\int_{G_0} \Phi_0 = \frac{1}{\Phi(0)} \, \int_{G_0} \Phi \le {\cal D}_{G_0}(V_{N + s} - V_{N + s}) = m_{G_0}(V_{N + s}).
\end{equation}

On the other hand, since $\supp{f_+} \subset G_0$,
\begin{equation} \label{phi-est-1}
\int_{G_0} \Phi = M^2 \int_{G_0} f \ge M^2 \int_{G} f > M^2 \, \al,
\end{equation}
and, since $(\Lambda - \Lambda) \cap \Omega_+ = \{0\}$,
\begin{equation} \label{phi-est-2}
\Phi(0) = M \, f(0) + \sum_{\lambda \in \Lambda'} \sum_{{\lambda' \in \Lambda'}\atop{\lambda' \ne \lambda}} f(\lambda - \lambda') \le M.
\end{equation}
Summarizing, we obtain from \eqref{eq:intPhizero}, (\ref{phi-est-1}), (\ref{phi-est-2}) and (\ref{M-below})
$$
m_{G_0}(V_{N + s}) \ge \frac{1}{\Phi(0)} \, \int_{G_0} \Phi \ge M \,  \al \ge (\rho - \varepsilon) \, m_{G_0}(V_N) \, \al
$$
so that
$$
\al \le \frac{1}{\rho - \varepsilon} \, \frac{m_{G_0}(V_{N + s})}{m_{G_0}(V_N)}.
$$
Finally, since $E$ is a tile, $m_{G_0}(V_N) = m_{G_0}(E + L_N) = (2N + 1)^d \, m_{G_0}(E)$, and
$$
\al \le \frac{1}{\rho - \varepsilon} \, \frac{(2N + 2s + 1)^d }{ (2N + 1)^d } .
$$
Taking limit when $N \to \infty$, we obtain
$$
\al \le \frac{1}{\rho - \varepsilon}.
$$
Letting $\ve\to 0$, we obtain $\al \le 1/\rho$ for all $\al< {\cal D}(\Omega_+)$, whence the desired statement follows.
 \eop 

Let $\Lambda$ be the center points of a ball packing with the unit ball $B$ in $\RR^d$. In accordance with~\eqref{stricpacking}, we have
$$
(B - B) \cap (\Lambda - \Lambda) = \{0\}.
$$
Since $B - B = 2B$, the strict packing type condition is fulfilled for $2B$ with the translation set $\Lambda$, and Theorem~\ref{main-theorem} yields
$$
\overline{D}^{\#}(\Lambda) \le \frac{1}{{\cal D}(2B)} =  \frac{1}{2^d {\cal D}(B)} .
$$
We arrive at the following statement.

\begin{corollary}\label{cor:centerdens} Let $\Lambda$ be the center points of a ball packing with the unit ball $B$ in $\RR^d$. Then for the density $\overline{\Delta}(\Lambda)$ of $\Lambda$ we have $\overline{\Delta}(\Lambda) \le \omega_d \overline{D}^{\#}(\Lambda) \le \frac{\omega_d}{2^d {\cal D}(B)}$, while for the center density $\overline{\delta}(\Lambda)$ we have $\overline{\delta}(\Lambda)\le \frac{1}{ 2^d{\cal D}(B)}$.
\end{corollary}

The same estimate holds for $B$ being replaced by any convex $0$-symmetric body in $\R^d$, by exactly the same argument, compare~\cite[Appendix~B]{cohn:packings}.

\begin{remark}
There are situations when the estimate in Theorem~\ref{main-theorem} is exact.
\end{remark}

This is, for example, the case, when  $\Omega_+ = H - H$ with $H \in \mathcal{B}_0$ which tiles $G$ in the strict sense with the translation set $\Lambda$.    In \cite[Corollary 5]{Revesz-2011} it was proved that in this case $\mathcal{T}(\Omega_+) = m_G(H)$ (see also Proposition~\ref{tile-theorem}). On the other hand, Corollary~\ref{density-corollary} gives $\overline{D}^{\#}(\Lambda) = \frac{1}{m_G(H)}$. Combining this with the result of Theorem~\ref{main-theorem} and noting that $\mathcal{T}(\Omega_+) \le \mathcal{D}(\Omega_+)$, we obtain
$$
\mathcal{D}(\Omega_+) = m_G(H) = \frac{1}{\overline{D}^{\#}(\Lambda)}.
$$

One further example, a very important one, is the result of Viazovska for the Euclidean ball $B$ in $\R^8$ \cite{Viaz}. In this situation the Delsarte  constant gives the \textit{exact} upper estimate for the density of any spherical packing,  which estimate is actually attained by the $E_8$ root lattice. Thus, once again, equality occurs in (\ref{main-theorem-estimate}) and in Corollary \ref{cor:centerdens}, too.

\bigskip

\noindent
Elena E. Berdysheva\\
University of Cape Town,\\
South Africa\\
{\tt elena.berdysheva@uct.ac.za}

\bigskip

\noindent
Szil\'ard Gy. R\'ev\'esz\\
Alfr\'ed R\'enyi Institute of Mathematics, \\
Budapest, Hungary\\
{\tt revesz.szilard@renyi.hu}

\end{document}